\documentclass[reqno]{amsart}
\usepackage{amsmath,amssymb,amsthm}
\usepackage{bm,bbm,comment, mathtools}
\usepackage{graphicx}
\usepackage[top=29truemm, bottom=29truemm,left=27.5truemm,right=27.5truemm]{geometry}

\makeatletter
\@namedef{subjclassname@2020}{%
  \textup{2020} Mathematics Subject Classification}
\makeatother
\thanks{This work was supported by JST SPRING, Grant Number JPMJSP2110. }
\author[N. Takeda]{Nobuki TAKEDA}
\address{Department of Mathematics, Graduate School of Science, Kyoto University, Kyoto 606-8502, Japan}
\email{takeda.nobuki.72z@st.kyoto-u.ac.jp}

\theoremstyle{definition}
\newtheorem{dfn}{Definition}
\newtheorem{rem}{Remark}
\theoremstyle{plain}
\newtheorem{prop}{Proposition}
\newtheorem*{cond}{Condition (A)}
\newtheorem{lem}[prop]{Lemma}
\newtheorem{thm}[prop]{Theorem}
\newtheorem{cor}[prop]{Corollary}
\newtheorem{fact}[prop]{Fact}
\newtheorem*{MainTheorem}{\rm\bf Theorem~\ref{thm:pull}}
\newtheorem*{MaindiffTheorem}{\rm\bf Corollary~\ref{cor:diff}}

\newcommand{\bbc}{\mathbb{C}}
\newcommand{\bbr}{\mathbb{R}}
\newcommand{\bbq}{\mathbb{Q}}
\newcommand{\bbz}{\mathbb{Z}}
\newcommand{\rmu}{\mathrm{U}}
\newcommand{\bfa}{\mathbf{a}}
\newcommand{\bfh}{\mathbf{h}}
\newcommand{\frakn}{\mathfrak{n}}
\newcommand{\bfi}{\boldsymbol{i}}
\newcommand{\calk}{\mathcal{K}}
\newcommand{\adele}{\mathbb{A}}
\newcommand{\kp}{K^+}
\newcommand{\dkl}{\mathbb{D}_{\boldsymbol{k}, \boldsymbol{l}}}
\renewcommand{\Re}{\mathrm{Re}}
\renewcommand{\Im}{\mathrm{Im}}
\newcommand{\bfk}{\boldsymbol{k}}
\newcommand{\bfl}{\boldsymbol{l}}
\newcommand{\kv}{\boldsymbol{k}_v}
\newcommand{\lv}{\boldsymbol{l}_v}
\newcommand{\va}{v\in\mathbf{a}}

\DeclareMathOperator*{\bigboxtimes}{ \boxtimes}

\providecommand{\abs}[1]{\left\lvert#1\right\rvert}
\providecommand{\norm}[1]{\left\lvert#1\right\rvert}
\providecommand{\adj}[1]{{#1^*}}
\providecommand{\uni}[1]{\mathrm{U}_{#1}}
\providecommand{\hus}[1]{\mathfrak{H}_{#1}}
\providecommand{\ph}[1]{\mathcal{P}_{#1}}
\providecommand{\inte}[1]{\mathcal{O}_{#1}}
\providecommand{\herm}[1]{\mathbb{S}_{#1}}
\DeclareMathOperator{\tr}{Tr}
\numberwithin{equation}{section}
\begin{document}
\title[Pullback formula for vector-valued hermitian modular forms on \(\rmu_{n,n}\)]
{Pullback formula for vector-valued hermitian modular forms on \(\rmu_{n,n}\)}
\date{}
\begin{abstract}
	Using a differential operator \(D\) which sends a scalar-valued Siegel modular form to the tensor product of two vector-valued Siegel modular forms,
	under a certain condition, we give the pullback formula for vector-valued hermitian modular forms on any CM field.
	We also give equivalence conditions for differential operators to have the above properties, which is an extension of Ibukiyama's result for hermitian modular forms.
\end{abstract}
\maketitle
\tableofcontents
\section{Introduction}
In the case of Siegel modular forms, the pullback of Siegel Eisenstein series to the direct product of two Siegel upper half-spaces has been studied by many people.
Garrett~\cite{Garrett1987Integral} proved the result in the scalar-valued case.
Using the method of Garrett~\cite{Garrett1987Integral} and differential operators that preserve the automorphic properties,
this theory was generalized for the symmetric tensor valued case (e.g. \cite{Bocherer1985Uber}, \cite{BoSaYa1992pullback}, \cite{Takei1992On}, \cite{Takayanagi1993Vector}, \cite{Kozima2000On}),
and for the alternating tensor valued case (e.g. \cite{Takayanagi1995On}, \cite{Kozima2002Standard}).
Kozima~\cite{Kozima2021Pullback} proved it in the general case by using the differential operators generalized by Ibukiyama~\cite{ibukiyama1999differential,Ibukiyama2020generic}.
In the case of hermitian modular forms, several results have been formulated (e.g. \cite{shimura2000arithmeticity}, \cite{LanUru2015Arithmeticity}, \cite{SkinnerUrban2014Iwasawa}).

This formula (called the ``pullback formula'') has been used for studying the Fourier
coefficients of vector-valued Klingen-Eisenstein series, the algebraicity of vector-valued Siegel modular forms and congruences of Siegel modular forms.

We first aim to describe differential operators on hermitian modular forms following Ibukiyama \cite{ibukiyama1999differential,Ibukiyama2020generic} using tools of representation theory such as Howe duality.
The case of  Siegel modular forms is reinterpreted in terms of representation theory in \cite{Takeda2025Kurokawa}, which is proved using the method of Ban \cite{ban2006rankin}.
A similar argument can be made for the case of hermitian modular forms as well.

Let \(K\) be a quadratic imaginary extension of a totally real field \(\kp\), and let \(\bfa\) (resp. \(\bfh\)) denote the set of infinite (resp. finite) places of \(\kp\).
We fix a CM type \(\Sigma_K\) of \(K\).
We put \(m:=m_{\kp}:=\#\bfa=[\kp:\bbq]\).
Let \(M_\rho(\Gamma_K^{(n)}(\frakn))\) be a complex vector space of all hermitian modular forms  of weight \(\rho\) and level \(\frakn\) on the product \(\hus{n}^\bfa\) of hermitian upper half planes.
We denote the representation of \(\mathrm{GL}_n(\bbc)\)
that corresponds to a dominant integral weight \(\bfk\) by \((\rho_{n, \bfk},V_{n,\bfk})\).
For a family \((\bfk,\bfl)=(\kv,\lv)_{\va}\) of pairs of dominant integral weights such that  \(\ell(\kv)\leq n\) and \(\ell(\lv)\leq n\) for any \(v\in \bfa\) (here \(\ell(\cdot)\) denotes the length of a dominant integral weight, defined in Section~\ref{sec:herm}),
we define the representation \(\rho_{n,(\bfk,\bfl)}=\bigboxtimes_{\va}(\rho_{n, \kv}\boxtimes\rho_{n,\lv})\)
of \(\calk_{n,\infty}^\bbc:=\prod_{\va}(\mathrm{GL}_n(\bbc)\times \mathrm{GL}_n(\bbc))\).
Let \(G_n\) be a unitary group defined over \(\kp\).

Let \(n_1,\ldots,n_d\) be positive integers such that \(n_1\geq\cdots\geq n_d\geq 1\) and put \(n = n_1+\cdots+n_d\).
Let \((\rho_s,V_s)\) be a representation of \(\calk_{(n_s)}^\bbc\) for \(s=1,\ldots,d\),
and \(\kappa=(\kappa_v)_\va\) a family of positive integers.
We consider \(V := V_{1}\otimes\cdots\otimes V_{d}\)-valued differential operators \(\mathbb{D}\)
on scalar-valued functions of \(\hus{n}\), satisfying Condition (A) below:

\begin{cond}
	For any modular form \(F\in M_\kappa(\Gamma_K^{(n)})\), we have
	\[\mathrm{Res}(\mathbb{D}(F))\in \bigotimes_{i=1}^d M_{\det^\kappa\rho_{n_i}}(\Gamma_K^{(n_i)}),\]
	where \(\mathrm{Res}\) means the restriction of a function on \(\hus{n}^\bfa\) to
	\(\hus{n_1}^\bfa \times\cdots\times \hus{n_d}^\bfa\)
	and \(\det^\kappa\) denotes the one-dimensional representation introduced in the Notation.
\end{cond}

This Condition (A) corresponds to Case (I) in \cite{ibukiyama1999differential},
and differential operators have been constructed for several vector-valued cases in \cite{Browning2024Constructing}.

We put \(\partial_Z=\left(\frac{\partial}{\partial Z_{v,i,j}}\right)_{\va}\).
Let \(P_v(X)\) be a vector-valued polynomial on a space \(M_n\) of degree \(n\) variable matrices.
We will give an equivalent condition that the differential operator \(\mathbb{D}=P(\partial_{Z})=(P_v(\partial_{Z_v}))_{\va}\) satisfies Condition (A).

\begin{MaindiffTheorem}
	Let \(n_1,\ldots,n_d\) be positive integers such that \(n_1\geq\cdots\geq n_d\geq 1\)
	and put \(n = n_1+\cdots+n_d\).
	We take a family \((\bfk_s, \bfl_s)=(\bfk_{v,s},\bfl_{v,s})_{\va}\) of pairs of dominant integral weights such that \(\ell(\bfk_{v,s})\leq n_d\), \(\ell(\bfl_{v,s})\leq n_d\)
	and \(\ell(\bfk_{v,s})+\ell(\bfl_{v,s})\leq\kappa_v\) for each \(\va\) and \(s=1,\ldots,d\).

	Let \(P_v(T)\) be a \(\left(V_{n_1,\bfk_{v,1},\bfl_{v,1}}\otimes\cdots\otimes V_{n_d,\bfk_{v,d},\bfl_{v,d}}\right)\)-valued polynomial on a space of degree \(n\) variable matrices \(M_n\) for \(\va\),
	and put \(P(T)=(P_v(T))_{\va}\).
	The differential operator \(\mathbb{D}=P(\partial_{Z})=(P_v(\partial_{Z_v}))_{\va}\) satisfies Condition (A) for \(\det^\kappa\)
	and \(\det^\kappa\rho_{n_1,(\bfk_{1},\bfl_1)}\otimes\cdots\otimes\det^\kappa\rho_{n_d,(\bfk_d,\bfl_d)}\)
	if and only if
	\(P(T)\) satisfies the following conditions:
	\begin{enumerate}
		\item If we put \(\widetilde{P}(X_1,\ldots,X_d,Y_1,\ldots,Y_d)=P\left(
			      \begin{pmatrix}
					      X_1\,^tY_1 & \cdots & X_1\,^tY_d \\
					      \vdots     & \ddots & \vdots     \\
					      X_d\,^tY_1 & \cdots & X_d\,^tY_d
				      \end{pmatrix}\right)\) with \(X_i, Y_i \in (M_{n_i,\kappa_v})_{\va}\),
		      then \(\widetilde{P}\) is pluriharmonic for each \((X_i, Y_i)\).
		\item For \((A_i,B_i) \in \calk_{(n_i)}^\bbc:=\prod_{\va}(\mathrm{GL}_{n_i}(\bbc)\times\mathrm{GL}_{n_i}(\bbc))\),
		      we have
		      \[P\left(
			      \begin{pmatrix} A_1 & & \\&\ddots&\\ & &  A_d\end{pmatrix}
			      T
			      \begin{pmatrix} ^t\!B_1 & & \\&\ddots&\\ & &  ^t\!B_d\end{pmatrix}
			      \right)
			      =\left(\rho_{n_1,(\bfk_1,\bfl_1)}(A_1,B_1)\otimes\cdots\otimes\rho_{n_d,(\bfk_d,\bfl_d)}(A_d,B_d)\right)P(T).\]
	\end{enumerate}
\end{MaindiffTheorem}

Then, we give the pullback formula for general vector-valued hermitian modular forms.
We use Shimura's result \cite{shimura2000arithmeticity} for calculations at the finite places
and we use Kozima's method \cite{Kozima2021Pullback} for the infinite places.

Let \(n_1, n_2\) be positive integers such that \(n_1\geq n_2\),
\(\kappa=(\kappa_v)_{\va }\) a family of positive integers and \(\bfk=(\bfk_v)_{\va }\) and \(\bfl=(\bfl_v)_{\va }\)
a family of dominant integral weights such that \(\ell(\bfk_v)\leq n_2\), \(\ell(\bfl_v)\leq n_2\)
and \(\ell(\bfk_v)+\ell(\bfl_v)\leq\kappa_v\) for each infinite place \(v\) of \(\kp\).
Let \(\frakn\) be an integral ideal of \(\kp\).
We take a Hecke character \(\chi\) of \(K\) with infinite part \(\chi_\infty\) given by
\[
	\chi_\infty(x_\infty)
	= \prod_{v \in \Sigma_K} |x_v|^{-\kappa_v} x_v^{\kappa_v},
\]
and with conductor dividing \(\frakn\).
Let \(E_{n,\kappa}(g,s;\frakn,\chi)\) on \(G_{n, \adele}\) be the hermitian Eisenstein series of degree \(n\), level \(\frakn\), and weight \(\kappa\),
and \(\left[f\right]_{r}^n(g,s;\chi)\) the hermitian Eisenstein series of degree \(n\) associated with an automorphic cusp form  \(f\) of degree \(r\).
Both series are defined explicitly in Section~\ref{sec:eisenstein}.

Let \(\chi_{K}\) be the quadratic character associated to quadratic extension \(K/\kp\).
For a Hecke eigenform \(f\) on \(G_{n,\adele}\) of level \(\frakn\), and weight \((\rho,V)\)
and a Hecke character \(\eta\) of \(K\), we set
\[D(s,f; \eta)
	=L(s-n+1/2,f\otimes \eta,\mathrm{St})
	\cdot\left(\prod_{i=0}^{2n-1} L_{\kp}(2s-i,\eta\cdot\chi_{K}^{i})\right)^{-1},
\]
where \(L(*,f\otimes {\eta}, \mathrm{St})\) is the standard L-function attached
to \(f\otimes {\eta}\) of degree \(2n + 1\), and \(L_{\kp}(*,\eta)\) (resp. \(L_{\kp}(*,\eta\cdot\chi_{K})\)) is the Hecke L-function
attached to \(\eta\) (resp. \(\eta\cdot\chi_{K}\)).
For a finite set \(S\) of finite places, we put
\[D_S(s,f;\eta)=\prod_{v\in\bfh\backslash S}D_v(s,f;\eta),\]
where \(D_v(s,f;\eta)\) is a \(v\)-part of \(D(s,f; \eta)\).

We put \(g^\natural=\left(\begin{smallmatrix}0&I_r\\I_r&0\\ \end{smallmatrix}\right)
	g\left(\begin{smallmatrix}0&I_r\\I_r&0\\ \end{smallmatrix}\right)\) and define \(f^\dagger_\chi\) by
\[
	f^\dagger_\chi(g)
	= \chi(\det g)\,\overline{f(g^\natural)},
\]
where \(\chi\) is a Hecke character of \(K\) whose archimedean component
\(\chi_\infty\) satisfies the above condition.

The main theorem is stated as follows.
Here, \((\cdot,\cdot)\) denotes the Petersson inner product, explicitly defined
with the usual normalization in Section~\ref{sec:herm}.
The map \(\iota\) denotes the embedding
\(G_{n_1} \times G_{n_2} \hookrightarrow G_n\) introduced in Section~\ref{sec:decomp}.
\begin{MainTheorem}
	Let \(S\) be the finite set of finite places dividing \(\frakn\),
	and take \(s \in \bbc\) such that \(\Re(s)>n\).
	\begin{enumerate}
		\item If \(n_1=n_2\), for any Hecke eigenform \(f\in \mathcal{A}_{0,n_2}(\rho_{n_2},\frakn)\), we have
		      \begin{align*}
			      \left(f,(\dkl E^\theta_{n,\kappa})(\iota(g_1,*),\overline{s}; \chi)\right)
			      =  c(s,\rho_{n_2}) \cdot \prod_{v\mid\frakn}[\calk_{n,v}:\calk_{n,v}(\frakn)]
			      \cdot D_S(s,f;\overline{\chi}) \cdot \overline{f^\dagger_\chi(g_1)}.
		      \end{align*}
		\item If \(\frakn=\inte{\kp}\),
		      for any Hecke eigenform \(f\in \mathcal{A}_{0,n_2}(\rho_{n_2})\),
		      we have
		      \begin{align*}
			      \left(f,(\dkl E_{n,\kappa})(\iota(g_1,*),\overline{s};\frakn, \chi)\right)
			      =  c(s,\rho_{n_2}) \cdot D(s,f;\overline{\chi}) \cdot \overline{[f^\dagger_\chi]_{n_2}^{n_1}(g_1,s; \chi)}.
		      \end{align*}
	\end{enumerate}
	Here a \(\bbc\)-valued function \(c(s,\rho_{n_2,v})\) is defined in Proposition~\ref{prop:arch}, which does not depend on \(n_1\).
\end{MainTheorem}

This paper is organized as follows:
In Section 2,
we explain hermitian modular forms and the terminology used in this paper.
In Section 3, we first give some formulas on derivatives.
Next, we give the equivalent condition for a differential operator on hermitian modular forms
to preserve the automorphic properties.
In Section 4,
we define the hermitian Eisenstein series and Klingen--Eisenstein series.
In Section 5, we first state the well-known fact of the double coset decomposition.
Then we show the calculation of the pullback formula can be reduced to a local computation
and prove the pullback formula.

\textbf{Acknowledgment.}
The author would like to express sincere gratitude to T. Ikeda for his guidance and support as the author's supervisor,
and to H. Katsurada for valuable suggestions and comments.
The author also thanks the reviewers for their detailed and constructive comments.

\textbf{Notation.}
For a commutative ring \(R\), we denote by \(R^\times\) the unit group of \(R\).
We denote by \(M_{m,n}(R)\) the set of \(m\times n\) matrices with entries in \(R\).
In particular, we put \(M_n(R)\coloneqq M_{n,n}(R)\).
Let \(I_n\) be the identity element of \(M_n(R)\).
Let \(\det(X)\) be the determinant of \(X\),
\(\tr(X)\) the trace of \(X\),
and \({}^tX\) the transpose of \(X\) for a square matrix \(X \in M_n(R)\).
Let \(\mathrm{GL}_n(R) \subset M_n(R)\) be a general linear group of degree \(n\).

Let \(K\) be a quadratic extension field of \(\kp\) with the non-trivial automorphism \(\tau\) of \(K\) over \(\kp\).
We often put \(\overline{k}=\tau(k)\) for \(k\in K\).
We put \(\overline{X}=(\overline{x_{ij}})\) and \(\adj{X}= {}^t\!\overline{X}\) for \(X=(x_{ij})\in M_{m,n}(K)\).

For matrices \(A\in M_{m}(\bbc), B\in M_{m,n}(\bbc)\), we define \(A[B]=\adj{B}AB\),
where \(\adj{B}\) is the transpose of \(\bar{B}\) and \(\bar{B}\) is the complex conjugate of \(B\).

Let \(\herm{n}\subset M_n(\bbc)\) be the set of hermitian matrices.
For an element \(X \in \herm{n}\),
we write \(X>0\) (resp. \(X \geq 0\)) if \(X\) is positive definite (resp. non-negative definite).
For a subset \(S \subset \herm{n}\), we denote by \(S_{>0}\) (resp. \(S_{\geq 0}\))
the subset of positive definite (resp. non-negative definite) matrices in \(S\).
Let \(\det^k\) be the one-dimensional representation of \(\mathrm{GL}_n(\bbc)\) given by \(g\mapsto\det(g)^k\).

Let \(K\) be a number field,
and let \(\mathfrak{p}\) be a prime ideal of \(K\).
We denote by \(K_\mathfrak{p}\) a \(\mathfrak{p}\)-adic completion of \(K\)
and by \(\inte{K}\) the ring of integers of \(K\).

If a group \(G\) acts on a set \(V\), then we denote by \(V^G\) the \(G\)-invariant
subspace of \(V\).

For a representation \((\rho,V)\), we denote by \((\rho^*, V^*)\) the
contragredient representation of \((\rho, V)\)
and by \((\overline{\rho},\overline{V})\) the conjugate representation of \((\rho, V)\).

\section{Hermitian Modular Forms}
Let \(K\) be a quadratic imaginary extension of a totally real field \(\kp\).
The set of finite places will be denoted by \(\bfh\) and the archimedean places by \(\bfa\).
We fix a CM type \(\Sigma_K\) for \(K\) (i.e. \(\Sigma_K\) contains a choice of exactly one representative from each pair of complex conjugate embeddings of \(K\)).
We often canonically identify CM type \(\Sigma_K\) with \(\bfa\), and we denote by \(\sigma_v\) the embedding \(K \hookrightarrow \bbc\) corresponding to \(v \in \bfa\).
We put \(m:=m_{\kp}:=\#\bfa=[\kp:\bbq]\).
We put \(K_v=\prod_{w|v} K_w\) and \(\mathcal{O}_{K_v}=\prod_{w|v}\mathcal{O}_{K_w}\) for a place \(v\) of \(\kp\).
Let \(\adele=\adele_{\kp}\) be the adele ring of \(\kp\), and \(\adele_{0}\),
\(\adele_{\infty}\) the finite and infinite parts of \(\adele\), respectively.

We put \(J_n=\begin{pmatrix}O_n & I_n \\ -I_n & O_n \\ \end{pmatrix}\).
The unitary group \(\uni{n}\) is an algebraic group defined over \(\kp\),
whose \(R\)-points are given by
\[\uni{n}(R)=\{g\in \mathrm{GL}_{2n}(K\otimes_{\kp}R) \mid \adj{g}J_ng=J_n\}\]
for each \(\kp\)-algebra \(R\).\\ We also define other unitary groups \(\rmu(n,n)\)
and \(\rmu(n)\) by
\begin{align*}
	\rmu(n,n) & =\{g\in \mathrm{GL}_{2n}(\bbc) \mid \adj{g}J_ng=J_n\}, \\
	\rmu(n)   & =\{g\in \mathrm{GL}_{n}(\bbc) \mid \adj{g}g =I_n\},
\end{align*}
where \(\adj{g}\) denotes the conjugate transpose of \(g\).
Put \(G_n=\rmu_n(\kp)\), \(G_{n,v}=\rmu_n(\kp_v)\)  for a place \(v\) of \(\kp\),
\(G_{n,\adele}=\rmu_n(\adele)\), \(G_{n,0}=\rmu_n(\adele_0)\),
and \(G_{n,\infty}=\prod_{\va}G_{n,v}=\prod_{\va}\rmu(n,n)\).

We define \(\calk_{n,v}\) by
\[\calk_{n,v}=\left\{
	\begin{array}{ll}
		\uni{n}(\inte{\kp_v}) & (v\in\bfh), \\
		\rmu(n)\times \rmu(n) & (\va).
	\end{array}
	\right.\]
Then \(\calk_{n,v}\) is isomorphic to a maximal compact subgroup of \(G_{n,v}\).
We fix a maximal compact subgroup of \(G_{n,v}\),
which is also denoted by \(\calk_{n,v}\) by abuse of notation.
We put \(\calk_{n,0}=\prod_{v\in\bfh}\calk_{n,v}\) and \(\calk_{n,\infty}=\prod_{\va }\calk_{n,v}\).

\subsection{As Analytic Functions on hermitian Symmetric Spaces}\label{sec:herm}
We have the identification
\begin{align*}
	M_n(\bbc) & \cong \herm{n}\otimes_\bbr\bbc  \\
	Z         & \mapsto \Re(Z)+\sqrt{-1}\Im(Z),
\end{align*}
with the hermitian real part \(\Re(Z)\) and the imaginary part \(\Im(Z)\), i.e.,
\begin{align*}
	\Re(Z) & =\frac{1}{2}(Z+\adj{Z}),          \\
	\Im(Z) & =\frac{1}{2\sqrt{-1}}(Z-\adj{Z}).
\end{align*}

Let \(\hus{n}\) be the hermitian upper half space of degree \(n\), that is
\[ \hus{n} = \left\{Z \in M_n(\bbc) \mid \Im(Z) >0\right\}.\]

Then \(G_{n,\infty}=\prod_{\va}\rmu(n,n)\) acts on \(\hus{n}^\bfa\) by
\[g\left<Z\right>=\left((A_vZ_v+B_v)(C_vZ_v+D_v)^{-1}\right)_{\va}\]
for  \(g=\begin{pmatrix}A_v& B_v \\C_v & D_v \\\end{pmatrix}_{\va} \in G_{n,\infty}\)
and \(Z=(Z_v)_{\va}\in\hus{n}^\bfa\).
We put \(\bfi_n:=(\sqrt{-1}I_n)_{\va}\in\hus{n}^\bfa\).

Let \((\rho, V )\) be an algebraic representation of
\(\calk_{n,\infty}^\bbc:=\prod_{\va}(\mathrm{GL}_n(\bbc)\times\mathrm{GL}_n(\bbc))\) on a finite dimensional complex vector space \(V\),
and take a hermitian inner product on \(V\) such that
\[\left<\rho(g)v,w\right>=\left<v,\rho(\adj{g})w\right> \]
for any \(g \in \calk_{n,\infty}^\bbc\).

For  \(g=\begin{pmatrix}A_v& B_v \\C_v & D_v \\\end{pmatrix}_{\va} \in G_{n,\infty}\)
and \(Z=(Z_v)_{\va}\in\hus{n}^\bfa\),
we put
\[\lambda(g,Z)=(C_vZ_v+D_v)_{\va}, \quad \mu(g,Z)=(\overline{C_v}{}^t\!Z_v+\overline{D_v})_{\va},\]
and
\[M(g,Z)=(\lambda(g,Z),\mu(g,Z)).\]
We write
\[\lambda(g)=\lambda(g,\bfi_n),\quad \mu(g)=\mu(g,\bfi_n)\ \text{ and }\ M(g)=M(g,\bfi_n)\]
for short.
For a \(V\)-valued function \(F\) on \(\hus{n}^\bfa\), we put
\[ F|_{\rho}[g](Z)=\rho(M(g,Z))^{-1}F(g\left< Z\right>) \quad (g\in G_{n,\infty},\ Z \in \hus{n}^\bfa).\]

We put
\[\Gamma_K^{(n)}(\frakn)=\left\{\left.g=(g_v)_{\va}\in G_{n,\infty}\cap \mathrm{GL}_{2n}(\inte{K})\right| g_v\equiv I_{2n} \mod{\frakn \mathcal{O}_K}\right\}\]
for an integral ideal \(\frakn\) of \(\kp\).
When \(\frakn=\inte{\kp}\), we put \(\Gamma_K^{(n)}=\Gamma_K^{(n)}(\inte{\kp})\).

\begin{dfn}
	We say that \(F\) is a (holomorphic) hermitian modular form of level \(\frakn\), and weight \((\rho,V)\)
	if \(F\) is a holomorphic \(V\)-valued function on \(\hus{n}\) and \(F|_{\rho}[g]=F\) for all \(g \in \Gamma_K^{(n)}(\frakn)\).
	(If \(n=1\) and \(\kp=\bbq\), another holomorphy condition at the cusps is also needed.)

	We denote by \(M_\rho(\Gamma_K^{(n)}(\frakn))\) a complex vector space of all hermitian modular forms
	of level \(\frakn\), and weight \((\rho,V)\).
\end{dfn}

We set
\[\Lambda_n(\frakn)=\left\{T\in\herm{n}\cap  M_n(K)\left|\sum_{\sigma\in\Sigma_K}(\tr(T_\sigma \xi_\sigma))\in \bbz \text{ for any } \xi\in \herm{n}\cap M_n(\frakn \mathcal{O}_K)  \right.\right\},\]
where \(T_{\sigma}\) is the image of \(T\in M_n(K)\) by the embedding corresponding to \(\sigma\in\Sigma_K\).
and let \(\Lambda_n(\frakn)_{\geq0}\) be a subset of \(\Lambda_n(\frakn)\) consisting of non-negative definite matrices (for all embeddings \(\sigma\in \Sigma_K\)).
Then, a modular form \(F \in M_\rho(\Gamma_K^{(n)}(\frakn))\) has the Fourier expansion
\[ F(Z)=\sum_{T\in \Lambda_n(\frakn)_{\geq0}}a(F,T)\mathbf{e}\left(\sum_{\va}(\tr(T_{\sigma_v} Z_v))\right),\]
where \(a(F,T)\in V\), \(\mathbf{e}(z)=\exp(2\pi\sqrt{-1}z)\).

If \(a(F, T)=0\) unless \(T\) is positive definite,
we say that \(F\) is a  (holomorphic) hermitian cusp form of level \(\frakn\), and weight \((\rho,V)\).
We also denote by \(S_\rho(\Gamma_K^{(n)}(\frakn))\) a complex vector space of all cusp forms of level \(\frakn\), and weight \((\rho,V)\).

Write the variable \(Z=(X_v+\sqrt{-1}Y_v)_{\va}\) on \(\hus{n}^\bfa\) with \(X_v,Y_v \in \herm{n}\) for each \(\va\). We
identify \(\herm{n}\) with \(\bbr^{n^2}\) and define measures \(dX_v, dY_v\) as the
standard measures on \(\bbr^{n^2}\). We define a measure \(dZ\) on \(\hus{n}^\bfa\) by
\[dZ= \prod_{\va}dX_v dY_v.\]
For \(F,G \in M_\rho(\Gamma_K^{(n)}(\frakn))\), we can define the Petersson inner product
as
\[ (F,G)=\int_{D } \left<\rho(Y^{1/2},{}^tY^{1/2})F(Z),\rho(Y^{1/2},{}^tY^{1/2})G(Z)\right>(\prod_{\va}\det(Y_v)^{-2n})dZ,\]
where \(Y=(Y_v)_{\va}=\Im(Z)\), \(Y^{1/2}=(Y^{1/2}_v)_{\va}\) is a family of positive definite hermitian matrices such that
\((Y_v^{1/2})^2=Y_v\), and \(D\) is a Siegel domain on \(\hus{n}^\bfa\)
for \(\Gamma_K^{(n)}(\frakn)\). This integral converges if either \(F\) or \(G\) is a cusp
form.

We call a sequence of non-negative integers \(\bfk=(k_1,k_2,\ldots)\) a dominant integral weight
if \(k_i \geq k_{i+1}\) for all \(i\), and \(k_i=0\) for almost all \(i\).
The largest integer \(m\) such that \(k_m \neq 0\) is called the length of \(\bfk\) and denoted by \(\ell(\bfk)\).
The set of dominant integral weights with length less than or equal to \(n\)
corresponds bijectively to the set of irreducible algebraic representations of \(\mathrm{GL}_n(\bbc)\).

For a family \((\bfk,\bfl)=(\kv,\lv)_{\va}\) of pairs of dominant integral weights
such that \(\ell(\kv)\leq n\) and \(\ell(\lv)\leq n\) for any \(v\in \bfa\),
we define the representation
\(\rho_{n,(\bfk,\bfl)}=\bigboxtimes_{\va}(\rho_{n, \kv}\boxtimes\rho_{n, \lv})\) of \(\calk_{n,\infty}^\bbc\).
We put \(M_{(\bfk,\bfl)}(\Gamma_K^{(n)}(\frakn))=M_{\rho_{n,(\bfk,\bfl)}}(\Gamma_K^{(n)}(\frakn))\)
and \(S_{(\bfk,\bfl)}(\Gamma_K^{(n)}(\frakn))=S_{\rho_{n,(\bfk,\bfl)}}(\Gamma_K^{(n)}(\frakn))\).
When \(\bfk=((\kappa_v,\ldots,\kappa_v))_{\va}\) and \(\bfl=((0,\ldots,0))_{\va}\)
for a family \(\kappa=(\kappa_v)_{\va}\) of non-negative integers,
we also put \(\det^\kappa=\rho_{n,(\bfk,\bfl)}\),
\(M_\kappa(\Gamma_K^{(n)}(\frakn))=M_{(\bfk,\bfl)}(\Gamma_K^{(n)}(\frakn))\)
and \(S_\kappa(\Gamma_K^{(n)}(\frakn))=S_{(\bfk,\bfl)}(\Gamma_K^{(n)}(\frakn))\).

\subsection{As Functions on \(\rmu(n,n)\).}\label{sec:functionsonU(n,n)}
Let \(\calk_{n,\infty}\) be the stabilizer of \(\bfi_n \in \hus{n}^\bfa\) in
\(G_{n,\infty}\). Then, \(\calk_{n,\infty}\) is a maximal compact subgroup of \(G_{n,\infty}\)
and isomorphic to \(\prod_{\va}\rmu(n)\times\rmu(n)\), which is given by
\[\begin{array}{rccc}
		 & \prod_{\va}\rmu(n)\times\rmu(n) & \rightarrow & \calk_{n,\infty}                                                                                           \\
		 & (k_{1,v},k_{2,v})_{\va}         & \mapsto     & \left(\mathfrak{c}\begin{pmatrix}k_{2,v}&0\\0&{}^tk_{1,v}^{-1}\end{pmatrix}\mathfrak{c}^{-1}\right)_{\va},
	\end{array}\]
where \(\mathfrak{c}=\dfrac{1}{\sqrt{2}}\begin{pmatrix}1 & \sqrt{-1}\\ \sqrt{-1} & 1\\ \end{pmatrix}\in M_{2n}(\bbc)\).
\begin{rem}
	Here we adopt this somewhat unusual isomorphism in Proposition~\ref{prop:automisom}
	to ensure consistency with the maximal compact subgroup \(\calk_{n,\infty}\) and the automorphic factors
	\(\lambda(g,Z)\) and \(\mu(g,Z)\).
	That is, for
	\[
		k=\left(\mathfrak{c}\begin{pmatrix}k_{2,v}&0\\0&{}^tk_{1,v}^{-1}\end{pmatrix}\mathfrak{c}^{-1}\right)_{\va} \in \calk_{n,\infty},
	\]
	we have
	\[
		\lambda(k,\bfi_n)=(k_{1,v})_{\va} \quad \text{and} \quad \mu(k,\bfi_n)=(k_{2,v})_{\va}.
	\]
\end{rem}

We put \(\mathfrak{g}_{n,v}=\mathrm{Lie}(G_{n,v})\),
\(\mathfrak{k}_{n,v}=\mathrm{Lie}(\calk_{n,v})\) and let \(\mathfrak{g}^\bbc_{n,v}\) and
\(\mathfrak{k}^\bbc_{n,v}\) be the complexification of \(\mathfrak{g}_{n,v}\) and
\(\mathfrak{k}_{n,v}\), respectively.
We have the Cartan decomposition \(\mathfrak{g}_{n,v}=\mathfrak{k}_{n,v}\oplus\mathfrak{p}_{n,v}\).
Furthermore, we put
\begin{align*}
	\kappa_{v,i,j}  & =\mathfrak{c}\begin{pmatrix}0 & 0\\ 0 &-e_{v,j,i}\\ \end{pmatrix}\mathfrak{c}^{-1},\quad
	\kappa'_{v,i,j}=\mathfrak{c}\begin{pmatrix} e_{v,i,j} & 0\\ 0 & 0\\ \end{pmatrix}\mathfrak{c}^{-1},                           \\
	\pi^{+}_{v,i,j} & =\mathfrak{c}\begin{pmatrix}0 & e_{v,i,j}\\ 0 & 0\\ \end{pmatrix}\mathfrak{c}^{-1},\quad \mathrm{and} \quad
	\pi^{-}_{v,i,j}=\mathfrak{c}\begin{pmatrix}0 & 0\\ e_{v,i,j} & 0\\ \end{pmatrix}\mathfrak{c}^{-1},
\end{align*}
where
\(e_{i,j} \in M_{n,n}(\bbc)\) is the matrix whose only non-zero entry is 1 in the \((i,j)\)-component.
\(\{\kappa_{v,i,j}\}\) is a basis of \(\mathfrak{k}^\bbc_{n,v}\).
Let \(\mathfrak{p}^+_{n,v}\) (resp. \(\mathfrak{p}^-_{n,v}\)) be the \(\bbc\)-span of
\(\{\pi^{+}_{v,i,j}\}\) (resp. \(\{\pi^{-}_{v,i,j}\}\)) in \(\mathfrak{g}^\bbc_{n,v}\).
We then put \(\mathfrak{g}_{n}=\prod_{\va}\mathfrak{g}_{n,v}\),
\(\mathfrak{k}^\bbc_{n}=\prod_{\va}\mathfrak{k}^\bbc_{n,v}\),
etc.

For a representation \((\rho, U_\rho)\) of \(\calk_{n,\infty}\), we define the representation \((\rho',U_{\rho'}\ (=U_\rho))\) by
\(\rho'(g_1,g_2)=\rho({}^tg_1^{-1},{}^tg_2^{-1})\), which is isomorphic to \(\rho^*\).

\begin{dfn}
	Let \((\rho, U_\rho)\) be an irreducible unitary representation of \(\calk_{n,\infty}\)
	and \(\Gamma_{n}\) a discrete subgroup of \(G_{n}\). We embed \(\Gamma_{n}\) diagonally into \(G_{n,\infty}\) and consider it as a subgroup of \(G_{n,\infty}\).
	Then, a hermitian modular form of type \(\rho\) for \(\Gamma_{n}\) is
	a \(U_{\rho'}\)-valued \(C^\infty\)-function \(\phi\) on \(G_{n,\infty}\) which satisfies the following conditions:
	\begin{enumerate}
		\item \(\phi(\gamma gk)=\rho'(k)^{-1}\phi(g)\) for \(k\in \calk_{n,\infty}\) and \(\gamma\in \Gamma_{n}\),
		\item \(\phi\) is annihilated by the right derivation of \(\mathfrak{p}_{n}^-\),
		\item \(\phi\) is of moderate growth.
	\end{enumerate}
\end{dfn}

We denote the space of moderate growth \(C^\infty\)-functions on \(G_{n,\infty}\)
which are invariant under left translation by \(\Gamma_{n}\) by \(C_\mathrm{mod}^\infty(\Gamma_{n}\backslash G_{n,\infty})\)
and the space consisting of all hermitian modular forms of type \(\rho\) for \(\Gamma_n\)
by \(\left[C_\mathrm{mod}^\infty(\Gamma_{n}\backslash G_{n,\infty})\otimes U_\rho^*\right]^{\calk_{n,\infty},\mathfrak{p}_{n}^-=0}\).

For \(f \in M_{\rho}(\Gamma_K^{(n)})\), we define a \(U_{\rho}\)-valued \(C^\infty\)-function \(\phi_f\) on \(G_{n,\infty}\) by
\[\phi_f(g)=(f|_{\rho} g )(\sqrt{-1})=\rho(M(g))^{-1}f(g\left<\bfi_n\right>)\]
for \(g\in G_{n,\infty}\).
Then,
we have the following proposition.

\begin{prop}[e.g. \cite{Eischen2024Automorphic}]\label{prop:automisom}
	The above correspondence \(f\mapsto \phi_f\) gives the isomorphism
	\[M_\rho(\Gamma_K^{(n)})\overset{\sim}{\longrightarrow}
		\left[C_\mathrm{mod}^\infty(\Gamma_K^{(n)}\backslash G_{n,\infty})\otimes U_{\rho'}\right]^{\calk_{n,\infty},\mathfrak{p}_{n}^-=0}.\]
\end{prop}

\subsection{As Functions on Unitary Groups over the Adeles.}
There is a unique compact open subgroup \(\calk_{n,0}(\frakn)\) of \(G_{n,0}\) such that
\[\Gamma_K^{(n)}(\frakn)=G_n\cap \calk_{n,0}(\frakn)\calk_{n,\infty}\]
for an integral ideal \(\frakn\) of \(\kp\).
We put \(\calk_{n,v}(\frakn)=\calk_{n,0}(\frakn)_v\).
We remark that \(\calk_{n,v}(\frakn)=\calk_{n,v}\) for a finite place \(v\nmid\frakn\).

\begin{dfn}
	A hermitian automorphic form on \(G_{n,\adele}\) of level \(\frakn\), and weight \((\rho,V)\) is defined to be a \(V\)-valued smooth function \(f\) on \(G_{n,\adele}\)
	such that it is left \(G_{n}\)-invariant, right \(\calk_{n,0}(\frakn)\)-invariant, right \((\calk_{n,\infty}, \rho)\)-equivariant,
	of moderate growth, and \(Z(\mathfrak{g})\)-invariant, where \(Z(\mathfrak{g})\) denotes the center of the complexified Lie algebra  \(\mathfrak{g}\) of \(\rmu(n,n)\).

	We denote by \(\mathcal{A}_n(\rho,\frakn)\) the complex vector space of hermitian automorphic forms on \(G_{n,\adele}\) of weight \(\rho\).
\end{dfn}

\begin{dfn}
	A hermitian automorphic form \(f\in\mathcal{A}_n(\rho,\frakn)\) is called a cusp form if
	\[\int_{N(\kp)\backslash N(\adele)}f(ng)dn=0\]
	for any \(g \in G_{n,\adele}\) and any unipotent radical \(N\) of each proper
	parabolic subgroup of \(\rmu_n\).

	We denote by \(\mathcal{A}_{0,n}(\rho,\frakn)\) the complex vector space of cusp
	forms on \(G_{n,\adele}\) of weight \(\rho\).
\end{dfn}

We put
\[ g_Z =\begin{pmatrix}
		Y^{1/2} & XY^{-1/2} \\ 0_n& Y^{-1/2}
	\end{pmatrix}\in G_{n,\infty}\]
for \(Z=X+\sqrt{-1}Y\in\hus{n}^\bfa\).
For \(f\in \mathcal{A}_n(\rho,\frakn)\), we define a function \(\hat{f}\) on
\(\hus{n}^\bfa\) by
\[\hat{f}(Z)=\rho(M(g_Z))f(g_Z).\]
Then, we have \(\hat{f}\in M_\rho(\Gamma_K^{(n)}(\frakn))\).
Moreover, if \(f \in	\mathcal{A}_{0,n}(\rho,\frakn)\),
then we have \(\hat{f}\in S_\rho(\Gamma_K^{(n)}(\frakn))\).

For \(v\in \bfh\), we take the Haar measure \(dg_v\) on \(G_{n,v}\) normalized so that the volume of \(\calk_{n,v}\) is 1.
For \(v\in \bfa\), we take the Haar measure \(dg_v\) on \(G_{n,v}\) such that the volume of \(\calk_{n,v}\) is 1
and the Haar measure on \(\hus{n}\cong G_{n,v}/\calk_{n,v}\) induced from \(dg_v\) is \((\det Y_v)^{-2n}dZ_v\).
Using these, we fix the Haar measure \(dg=\prod_vdg_v\) on \(G_{n,\adele}\).
We define the Petersson inner product on \(\mathcal{A}_n(\rho,\frakn)\) as
\[ (f,h)=\int_{G_{n}\backslash G_{n,\adele}} \left<f(g),h(g)\right>dg,\]
for \(f,h \in \mathcal{A}_n(\rho,\frakn)\), where \(dg\) is a Haar measure on \(G_{n}\backslash G_{n,\adele}\)
induced from that on \(G_{n,\adele}\).

For a finite place \(v\in\bfh\) that corresponds to a prime ideal \(\mathfrak{p}\) of \(\kp\),
let \(\mathcal{H}_{n,\mathfrak{p}}\) be the convolution algebra of
left and right \(\calk_{n,v}\)-invariant compactly supported \(\bbq\)-valued functions of \(G_{n,v}\),
which is called the spherical Hecke algebra at \(\mathfrak{p}\).
The spherical Hecke algebra \(\mathcal{H}_{n,\mathfrak{p}}\) at \(v\) acts
on the set of continuous right \(\calk_{n,v}\)-invariant functions on \(G_{n,v}\) (or on \(G_{n,\adele}\)) by right convolution,
i.e., for a continuous right \(\calk_{n,v}\)-invariant function \(f\) on \(G_{n,v}\) (or on \(G_{n,\adele}\)) and \(\eta\in\mathcal{H}_{n,\mathfrak{p}}\), we put
\[(\eta\cdot f)(g) =\int_{G_{n,v}}f(gh^{-1})\eta(h)dh,\]
where \(dh\) is a Haar measure on \(G_{n,v}\) normalized so that the volume of \(\calk_{n,v}\) is 1.

\begin{dfn}
	We say that a continuous right \(\calk_{n,v}\)-invariant function \(f\) on \(G_{n,v}\) (or on \(G_{n,\adele}\)) is a \(\mathfrak{p}\)-Hecke eigenfunction
	if \(f\) is an eigenfunction under the action of \(\mathcal{H}_{n,\mathfrak{p}}\).
\end{dfn}

\begin{dfn}
	We say that a hermitian automorphic form \(f\in\mathcal{A}_n(\rho,\frakn)\) is a Hecke eigenform
	if \(f\) is a \(\mathfrak{p}\)-Hecke eigenfunction for any \(\mathfrak{p}\) not dividing \(\frakn\).
\end{dfn}

\section{Differential Operators}
We put \(G_{(n)}=G_{n,\infty}\) and \(\calk_{(n)}=\calk_{n,\infty}\) as a symbol only for this section.
Furthermore, when \(n\) is obvious, we write \(G = G_{(n)}\) and \(K=\calk_{(n)}\) for short.

\subsection{Formulas on Derivatives}
Following Ibukiyama \cite{Ibukiyama2022Differential}, Ibukiyama-Zagier
\cite{Ibukiyama2014Higher}, and others, we will provide some formulas.

\begin{lem}\label{lem:det} For a positive integer \(d\) and \(Z \in \hus{n}\), we have
	\[\det\left(\frac{Z}{\sqrt{-1}}\right)^{-d}
		=(2\pi)^{-dn}\int_{M_{n,d}(\bbc)}\exp\left(\frac{\sqrt{-1}}{2}\tr(\adj{X}ZX)\right)dX,\]
	where \(dX\) is the Lebesgue measure on \(M_{n,d}(\bbc)\cong\bbr^{2nd}\). Moreover,
	if \(d\geq n\), then this is equal to
	\[c_n(d)\int_{{\herm{n}}_{>0}}\exp\left(\frac{\sqrt{-1}}{2}\tr(TZ)\right)(\det T)^{d-n}dT,\]
	where
	\(c_n(d)=2^{-dn}\pi^{-\frac{n(n-1)}{2}}\left(\prod_{i=0}^{n-1}\Gamma(d-i)\right)^{-1}\)
	and \(dT\) is the Lebesgue measure on \(\herm{n}\cong\bbr^{n^2}\).
\end{lem}
\begin{proof}
	This is a well-known fact and not difficult to prove, but we will provide the proof for the reader.

	Since these equations are holomorphic on \(Z\in\hus{n}\), it is sufficient to show
	when \(\Re(Z)=0\). In this case, we can write \(Z=\sqrt{-1}Y\) with a positive
	definite hermitian matrix \(Y\). There exists a positive definite hermitian matrix
	\(A\) such that \(Y=A^2\). Then, we have
	\begin{align*}
		\int_{M_{n,d}(\bbc)}\exp\left(\frac{\sqrt{-1}}{2}\tr(\:\adj{X}ZX)\right)dX
		 & =\int_{M_{n,d}(\bbc)}\exp\left(-\frac{1}{2}\tr(\:\adj{X}YX)\right)dX             \\
		 & =\int_{M_{n,d}(\bbc)}\exp\left(-\frac{1}{2}\tr(\:\adj{(AX)}(AX))\right)dX        \\
		 & =\det(A)^{-2d}\int_{M_{n,d}(\bbc)}\exp\left(-\frac{1}{2}\tr(\:\adj{X}X)\right)dX \\
		 & =\det(Y)^{-d}(2\pi)^{dn}
	\end{align*}
	Thus, the first equation holds.

	Now assume \(d\geq n\). We decompose \(X=\:^t(x_1\cdots x_n)\in M_{n,d}(\bbc)\) as
	\(X=LQ\) by a lower triangular matrix \(L=(l_{i,j})\in M_{n}(\bbc)\) with positive
	real diagonal components and \(Q=\:^t(v_1\cdots v_n)\in M_{n,d}(\bbc)\) such that
	\(\adj{Q}\,Q=I_n\). Let \(d\mu_n\) be a standard measure of the \(n\)-sphere \(S^n\). Then
	we have
	\begin{align*}
		dx_1 & =d(l_{1,1}v_1)=l_{1,1}^{2d-1}dl_{1,1}d\mu_{2d-1},                    \\
		dx_2 & =d(l_{2,1}v_1+l_{2,2}v_2)=l_{2,2}^{2d-3}dl_{2,1}dl_{2,2}d\mu_{2d-3}, \\
		     & \vdots
	\end{align*}
	We note that \(dl_{i,i}\) is a Lebesgue measure on \(\bbr\), but \(dl_{i,j} \ (i\neq j)\) is a Lebesgue measure on \(\bbc\).
	Multiply all of the above equations together to obtain
	\[dX=\prod_{i=1}^{n}l_{i,i}^{2d+1-2i}\prod_{i,j}dl_{i,j}dQ,\]
	where \(dQ=d\mu_{2d-1}d\mu_{2d-3}\cdots d\mu_{2d-2n+1}\). If we put
	\(T=X\adj{X}=L\adj{L}\), we can calculate that
	\[dT:=\prod_{i\leq j}dt_{i,j}=2^n\prod_{i=1}^{n}l_{i,i}^{2n+1-2i}\prod_{i,j}dl_{i,j}.\]
	Combining these equations gives
	\[dX=2^{-n}\prod_{i=1}^{n}l_{i,i}^{2d-2n}dTdQ=2^{-n}(\det T)^{d-n}dTdQ.\]
	From the above, we obtain
	\begin{align*}
		\int_{M_{n,d}(\bbc)} & \exp\left(-\frac{1}{2}\tr(\adj{X}YX)\right)dX                                                                                                    \\
		                     & = 2^{-n}\prod_{i=0}^{n-1}\mathrm{vol}(S^{2d-2i-1})\cdot\int_{{\herm{n}}_{>0}}\exp\left(-\frac{1}{2}\tr(TY)\right)(\det T)^{d-n}dT                \\
		                     & = \pi^{n(2d-n+1)/2}\left(\prod_{i=0}^{n-1}\Gamma(d-i)\right)^{-1}\cdot\int_{{\herm{n}}_{>0}}\exp\left(-\frac{1}{2}\tr(TY)\right)(\det T)^{d-n}dT
	\end{align*}
	Thus, the second equation holds.
\end{proof}

To simplify the notation, multi-variable functions and operations are often denoted as a single variable, e.g.,
\(X Y=(X_v Y_v)_{\va}\), \({}^tX=({}^tX_v)_{\va}\), \(\det(X)^\kappa=\prod_{\va}\det(X_v)^{\kappa_v}\), and  \(\tr(X)=\sum_{\va}\tr(X_v)\).
We set the functions
\begin{align*}
	\delta_g(Z)             & = \det(CZ+D),                                                 \\
	\Delta_g(Z)             & =(\Delta_g(Z)_v)_{\va}=((CZ+D)^{-1}C),                        \\
	\varrho_g(Z; \kappa ,s) & =\abs{\det(CZ+D)}^{\kappa-2s}\det(CZ+D)^{-\kappa },           \\
	\delta(g)               & =\delta_g(\bfi_n),                                            \\
	\Delta(g)               & =(\Delta(g)_v)_{\va}=((C\bfi_n+D)^{-1}(C+D\bfi_n)),           \\
	\varrho(g; \kappa ,s)   & = \abs{\det(C\bfi_n+D)}^{\kappa-2s}\det(C\bfi_n+D)^{-\kappa } \\
\end{align*}
for \(g=\begin{pmatrix}A&B\\C&D\end{pmatrix}_{\va} \in G\), \(Z=(Z_{v,i,j})_{\va}\in \hus{n}^\bfa\),
a family \(\kappa=(\kappa_v)_{\va}\) of positive integers and a complex variable \(s\).
By a direct calculation, we obtain the following formula.
\begin{lem}\label{lem:deri}
	We have
	\begin{align*}
		\frac{\partial}{\partial Z_{v,i,j}}\delta_g(Z)             & =\delta_g(Z)\Delta_g(Z)_{v,j,i},                   \\
		\frac{\partial}{\partial Z_{v,i,j}}\delta (g,Z)^{-\kappa } & =-\kappa \delta_g(Z)^{-\kappa}\Delta_g(Z)_{v,i,j}, \\
		\frac{\partial}{\partial Z_{v,i,j}}\Delta_g(Z)_{v,s,t}     & =-\Delta_g(Z)_{v,s,i}\Delta_g(Z)_{v,j,t}.
	\end{align*}
	Similarly, we have
	\begin{align*}
		\pi^+_{v,i,j}\cdot\delta(g)            & =\delta(g)\Delta(g)_{v,j,i},                  \\
		\pi^+_{v,i,j}\cdot\delta (g)^{-\kappa} & =-\kappa\delta(g)^{-\kappa}\Delta(g)_{v,i,j}, \\
		\pi^+_{v,i,j}\cdot\Delta(g)_{v,s,t}    & =-\Delta(g)_{v,s,i}\Delta(g)_{v,j,t}.
	\end{align*}
	Here, \(\pi^+_{v,i,j}\) is defined in Section~\ref{sec:functionsonU(n,n)}.
	In particular, we have
	\begin{align*}
		\frac{\partial}{\partial Z_{v,i,j}}(\varrho_g(Z;\kappa,s) ) & =\varrho_g(Z;\kappa,s) \left(-\frac{\kappa}{2}-s\right)\Delta_g(Z)_{v,j,i}, \\
		\pi^+_{v,i,j}(\varrho(g;\kappa,s))                          & =\varrho(g;\kappa,s) \left(-\frac{\kappa}{2}-s\right)\Delta(g)_{v,j,i}.
	\end{align*}
\end{lem}

We put \(\partial_Z=\left(\frac{\partial}{\partial Z_{v,i,j}}\right)_{\va}\) and \(\pi^+=\left(\pi^+_{v,i,j}\right)_{\va}\).
As a simple consequence, we obtain the following lemma.
\begin{lem}\label{lem:difpoly}
	\begin{enumerate}
		\item For a polynomial \(P(T)\in\bbc[T]\) with a family \(T=(T_v)_{\va}\) of degree \(n\) matrices of variables and
		      \(\kappa=(\kappa_v)_{\va} \in(\bbz\geq 0)^\bfa\), there is a polynomial \(Q(T;\kappa,s)\in\bbc[T]\) such that
		      \begin{align*}
			      P(\partial_Z)\varrho_g(Z;\kappa,s) & =\varrho_g(Z;\kappa,s) Q(\Delta_g(Z);\kappa,s), \\
			      P(\pi^+)\varrho(g;\kappa,s)        & =\varrho(g;\kappa,s)Q(\Delta(g);\kappa,s).
		      \end{align*}

		\item The polynomial \(Q\) in (1) also satisfies
		      \begin{align*}
			      P(\partial_Z)\delta (g,Z)^{-(\kappa/2+s)} & =\delta (g,Z)^{-(\kappa/2+s)} Q(\Delta_g(Z);\kappa,s), \\
			      P(\pi^+)\delta(g)^{-(\kappa/2+s)}         & =\delta(g)^{-(\kappa/2+s)}Q(\Delta(g);\kappa,s).
		      \end{align*}
	\end{enumerate}

\end{lem}

From now on, we assume \(\kappa_v\geq n\) for each \(\va\) in this section.
Let \(P(T)\in \bbc[T]\) be a homogeneous polynomial of degree \(\nu\).
From Lemma~\ref{lem:det}, we have
\begin{equation}\label{eq:PdZ}
	P(\partial_Z)\det\left(\frac{Z}{\sqrt{-1}}\right)^{-\kappa}
	=c_n(\kappa)^{m}\left(\frac{\sqrt{-1}}{2}\right)^\nu
	\int_{({\herm{n}}_{>0})^\bfa}\exp\left(\frac{\sqrt{-1}}{2}\tr(TZ)\right)\, {}^t\!P(T)(\det T)^{\kappa-n}dT,
\end{equation}
where \({}^t\!P(T)=P(^tT)\).

\begin{dfn}
	For a homogeneous polynomial \(P(T)\in \bbc[T]\), we define the function \(\mathcal{L}_\kappa(P)\) on \(({\herm{n}}_{>0})^\bfa\) as
	\[\mathcal{L}_\kappa(P)(Y)=\int_{({\herm{n}}_{>0})^\bfa}\exp\left(-\frac{1}{2}\tr(TY)\right)\, {}^t\!P(T)(\det T)^{\kappa-n}dT\]
	for \(Y \in ({\herm{n}}_{>0})^\bfa\).
\end{dfn}

By Lemma~\ref{lem:difpoly}, there exists a homogeneous polynomial
\(Q(T)\in\bbc[T]\) such that
\begin{equation} \label{eq:LtoQ}
	\mathcal{L}_\kappa(P)(Y)=(\det Y)^{-\kappa}Q(Y^{-1}).
\end{equation}

We take a family \(A=(A_v)_{\va}\in \herm{n}^\bfa\) of hermitian matrices such that \(Y_v=A_v^2\). If we put
\(T_1=(A_vT_vA_v)_{\va}\) and \(X_1=(A_vX_v)_{\va}\), using the notation in the proof of Lemma~\ref{lem:det},
we have
\[dX_1=2^{-nm }(\det T_1)^{\kappa-n}dT_1dQ=2^{-nm }(\det A)^{2\kappa-2n}(\det T)^{\kappa-n}dT_1dQ\]
and
\[dX_1=(\det A)^{2\kappa_v}dX=2^{-nm }(\det A)^{2\kappa}(\det T)^{\kappa-n}dTdQ.\]
From these equations, we have
\[dT_1=(\det A)^{{2n}}dT.\]
Therefore, we obtain
\[\mathcal{L}_\kappa(P)(Y)=\int_{({\herm{n}}_{>0})^\bfa}\exp\left(-\frac{1}{2}\tr(ATA)\right)\, {}^t\!P(T)(\det T)^{\kappa-n}dT
	=(\det Y)^{-\kappa}\mathcal{L}_\kappa(P_{A^{-1}})(I_n),\]
where \(P_{A^{-1}}(T)=P(A^{-1}TA^{-1})\). Thus, \(Q(Y^{-1})\) in \eqref{eq:LtoQ} is
equal to \(\mathcal{L}_\kappa(P_{A^{-1}})(I_n)\).

For \(X=(x_{v,i,j})_{\va}\in (M_{n,n})^\bfa\) and \(\nu=(\nu_{v,i,j})\in (M_n(\bbz_{\geq0}))^\bfa\), we put
\[\nu!=\prod_{v,i,j}\nu_{v,i,j}!\ , \quad x^\nu=\prod_{v,i,j}x_{v,i,j}^{\nu_{v,i,j}}\ , \quad \deg(\nu)=\sum_{v,i,j}\nu_{v,i,j}\]
and
\[E_\kappa[P]=c_n(\kappa)^{m}\mathcal{L}_\kappa(\, ^tP)(I_n).\]

\begin{lem}\label{lem:Eksum}
	Using the above notation, we have
	\[\sum_{\nu\in (M_n(\bbz_{\geq0}))^\bfa}E_\kappa[T^\nu]\frac{Y^\nu}{\nu!}=(\det(I_n-2Y))^{-\kappa}\]
	for a family \(Y\) of hermitian matrices of variables.
\end{lem}
\begin{proof}
	By definition, we have
	\[\sum_{\nu\in (M_n(\bbz_{\geq0}))^\bfa}E_\kappa[T^\nu]\frac{Y^\nu}{\nu!}
		=c_n(\kappa)^{m}\sum_{\nu\in (M_n(\bbz_{\geq0}))^\bfa}\int_{({\herm{n}}_{>0})^\bfa}\exp\left(-\frac{1}{2}\tr(T)\right)\frac{T^\nu Y^\nu}{\nu!}(\det T)^{\kappa-n}dT. \]
	If we assume \(I_n-Y\,\adj{Y}>0\), this is equal to
	\[c_n(\kappa)^{m}\int_{({\herm{n}}_{>0})^\bfa}\exp\left(-\frac{1}{2}\tr(T(I_n-2Y))\right)(\det T)^{\kappa-n}dT.\]
	In addition, if we put \(I_n-2Y=U^2\) with a family \(U\) of hermitian matrices and set
	\(T_1=UTU\), this is equal to
	\[(\det U)^{-2\kappa}c_n(\kappa)^{m}\int_{({\herm{n}}_{>0})^\bfa}\exp\left(-\frac{1}{2}\tr(T_1)\right)(\det T_1)^{\kappa-n}dT_1=(\det(I_n-2Y))^{-\kappa}. \]
	The last equation is due to Lemma~\ref{lem:det}.
\end{proof}

\begin{thm}\label{thm:PtoPhi}
	Let \(P(T)\in \bbc[T]\) be a homogeneous polynomial of degree \(d\) with a family of degree \(n\) matrices \(T=(T_v)_{\va}\)
	and \(\kappa=(\kappa)_{\va}\) a family of positive integers.
	If \(\kappa_v\geq n\) for each \(\va\), we have
	\[P(\partial_Z)(\delta_g(Z)^{-\kappa})=\delta_g(Z)^{-\kappa}\phi_\kappa(P)(\Delta_g(Z)),\]
	where
	\[\phi_\kappa(P)(T)=\left.\left(-1\right)^{dm}\left(P(\partial_W)\det(I_n-W\, ^tT)^{-\kappa}\right)\right|_{W=0}\]
	for \(g\in G\) and \(Z\in\hus{n}^\bfa\).

	In particular, if \(\kappa_v/2+s\geq n\) for each \(\va\), we have
	\[P(\pi^+)\varrho(g;\kappa,s)=\varrho(g;\kappa,s)\psi_{\kappa,s}(P)(\Delta(g);\kappa,s),\]
	where
	\[\psi_{\kappa,s}(P)(T;\kappa,s)=\left.(-1)^{dm}\left(P(\partial_W)\det(I_n-W\, ^tT)^{-(\kappa/2+s)}\right)\right|_{W=0}.\]
\end{thm}
\begin{proof}
	From Lemma~\ref{lem:deri}, \(\phi_\kappa(P)\) does not depend on the choice of \(g\) and \(Z\), and \(\kappa\) can be regarded as a variable.
	Therefore, it is sufficient to show the case \(g=\begin{pmatrix}0&I_n\\-I_n&0\end{pmatrix}\), \(Z=\sqrt{-1}Y\) with a positive definite hermitian matrix \(Y\).
	We take a family \(A\) of hermitian matrices such that \(Y=A^2\).
	We define the constants \(r_{\nu,\mu}(A)\) by
	\[(A^{-1}TA^{-1})^\nu=\sum_{\mu\in (M_n(\bbz_{\geq0})^\bfa)}r_{\nu,\mu}(A)\frac{T^\mu}{\mu!}\]
	for \(\nu\in (M_n(\bbz_{\geq0}))^\bfa\).
	We put \(P(T)=\sum_\nu c_\nu T^\nu\).
	Then, we have
	\begin{equation}\label{eq:thm1}
		c_n(\kappa)^{m}\det(Y)^\kappa\mathcal{L}_\kappa(P)(Y)=E_\kappa[\, ^t(P_{A^{-1}})]
		=\sum_{\nu,\mu} c_\nu r_{\nu,\mu}(A)\frac{E_\kappa[(\, ^tT)^\mu]}{\mu!}.
	\end{equation}
	On the other hand, since \(\left.(\partial_W)^\nu(W^\mu)\right|_{W=0}=\delta_{\nu, \mu}\nu!\),
	where \(\delta_{\nu, \mu}\) is the Kronecker delta, we have
	\begin{align}
		P(\partial_W)\left.\left(\det(I_n-2\,{}^t\!A^{-1}W\,{}^t\!A^{-1})^{-\kappa}\right)\right|_{W=0}
		 & =P(\partial_W)\left.\left(\det(I_n-2A^{-1}\, ^tWA^{-1})^{-\kappa}\right)\right|_{W=0} \notag                       \\
		 & =P(\partial_W)\left.\left(\sum_\mu E_\kappa[T^\mu]\frac{(A^{-1}\, ^tWA^{-1})^\mu}{\mu!}\right)\right|_{W=0} \notag \\
		 & =\sum_{\nu,\mu} c_\nu r_{\mu,\nu}(A)\frac{E_\kappa[T^{(^t\mu)}]}{(^t\mu)!} \notag                                  \\
		 & =\sum_{\nu,\mu} c_\nu r_{\mu,\nu}(A)\frac{E_\kappa[(^tT)^\mu]}{\mu!} \label{eq:thm2}
	\end{align}
	by Lemma~\ref{lem:Eksum}.
	Here the following lemma holds.
	\begin{lem}\label{lem:rmn} We have
		\(r_{\nu,\mu}(A)=r_{\mu,\nu}(A)\)
		for any \(\nu,\mu\in  (M_n(\bbz_{\geq0}))^\bfa\).
	\end{lem}
	\begin{proof}
		We define the inner product on \(\bbc[T]\) by
		\[(P(T),Q(T)) =(P(\partial_T)\overline{Q})(0)\]
		for \(P(T), Q(T)\in \bbc[T]\).
		For \(X, Y\in \mathrm{GL}_n(\bbc)\), we have
		\[(P(XTY), Q(T))=(P(T),Q(\adj{X}T\,\adj{Y})).\]
		Then, \(P(T)\mapsto P(A^{-1}TA^{-1})\) is self-adjoint with respect to this inner	product.
		Thus, the claim follows from the fact that
		\(\left\{\dfrac{T^\nu}{\sqrt{\nu!}}\mid\nu\in (M_n(\bbz_{\geq0}))^\bfa\right\}\) is an orthonormal basis of \(\bbc[T]\).
	\end{proof}
	Continuing the proof of Theorem~\ref{thm:PtoPhi}.
	From \eqref{eq:thm1}, \eqref{eq:thm2} and Lemma~\ref{lem:rmn}, we have
	\begin{align*}
		c_n(\kappa)^{m}\det(Y)^\kappa \mathcal{L}_\kappa(P)(Y) & =P(\partial_W)\left.\left(\det(I_n-2\, ^tA^{-1}W\, ^tA^{-1})^{-\kappa}\right)\right|_{W=0} \\
		                                                       & =P(\partial_W)\left.\left(\det (I_n-2W\,^tY^{-1})^{-\kappa}\right)\right|_{W=0}
	\end{align*}
	Since \(P\) is homogeneous of degree \(d\), we have
	\[P(\partial_W)\left.\left(\det (I_n-2\sqrt{-1}W\,^tZ^{-1})^{-\kappa}\right)\right|_{W=0}=(2\sqrt{-1})^dP(\partial_W)\left.\left(\det (I_n-W\,^tZ^{-1})^{-\kappa}\right)\right|_{W=0}\]
	Thus, Substituting for \eqref{eq:PdZ} we obtain the result.
\end{proof}

\begin{cor}\label{cor:acttophi}
	Let \(P(T)\in \bbc[T]\) be a homogeneous polynomial (of degree \(d\)). We put \(P_{A,B}(T)=P(\,^t\!ATB)\) for \(A,B \in (\mathrm{GL}_n(\bbc))^\bfa\).
	Then, we have
	\[\phi_\kappa(P_{A,B})(T)=\phi_\kappa(P)(\,^t\!ATB) \quad \text{and} \quad \psi_{\kappa,s}(P_{A,B})(T)=\psi_{\kappa,s}(P)(\,^t\!ATB)\]
	for \(\phi_\kappa\) and \(\psi_{\kappa,s}\) in Theorem~\ref{thm:PtoPhi}.
\end{cor}
\begin{proof}
	From the above theorem, we have
	\begin{align*}
		\phi_\kappa(P_{A,B})(T) & =\left.\left(-1\right)^{-dm}\left(P(\,^t\!A\partial_WB)\det(I_n-2W\, ^tT)^{-\kappa}\right)\right|_{W=0}   \\
		                        & =\left.\left(-1\right)^{-dm}\left(P(\partial_W)\det(I_n-2AW\,^t\!B\, ^tT)^{-\kappa}\right)\right|_{W=0}   \\
		                        & =\left.\left(-1\right)^{-dm}\left(P(\partial_W)\det(I_n-2W\,^t\!B\, ^tTA)^{-\kappa}\right)\right|_{W=0}   \\
		                        & =\left.\left(-1\right)^{-dm}\left(P(\partial_W)\det(I_n-2W\, ^t(\,^t\!ATB))^{-\kappa}\right)\right|_{W=0} \\
		                        & =\phi_\kappa(P)(\,^t\!ATB).
	\end{align*}
	The same can be done for \(\psi_{\kappa,s}\).
\end{proof}

\subsection{Differential Operators on Automorphic Forms}
Let \(n_1,\ldots,n_d\) be positive integers such that \(n_1\geq\cdots\geq n_d\geq 1\)
and put \(n = n_1+\cdots+n_d \geq 2\). We embed \(\hus{n_1}^\bfa \times\cdots\times \hus{n_d}^\bfa\) in \(\hus{n}^\bfa\) and
\(G_{(n_1)}\times\cdots\times G_{(n_d)}\) in \(G_{(n)}\) diagonally.

Let \((\rho_s,V_s)\) be a representation of \(\calk_{(n_s)}^\bbc\) for \(s=1,\ldots,d\),
and \(\kappa=(\kappa_v)_\va\) a family of positive integers.

We will consider \(V := V_{1}\otimes\cdots\otimes V_{d}\)-valued differential operators \(\mathbb{D}\)
on scalar-valued functions of \(\hus{n}^\bfa\), satisfying Condition (A) below:

\begin{cond}
	For any modular form \(F\in M_\kappa(\Gamma_K^{(n)})\), we have
	\[\mathrm{Res}(\mathbb{D}(F))\in \bigotimes_{i=1}^d M_{\det^\kappa\rho_{n_i}}(\Gamma_K^{(n_i)}),\]
	where \(\mathrm{Res}\) means the restriction of a function on \(\hus{n}^\bfa\) to
	\(\hus{n_1}^\bfa \times\cdots\times \hus{n_d}^\bfa\).
\end{cond}

\begin{rem}
	\begin{enumerate}
		\item Such differential operators have been constructed for several vector-valued cases in \cite{Browning2024Constructing}.
		\item Using the method of Ban \cite{ban2006rankin}, representation-theoretic interpretation of the differential operators
		      satisfying Condition (A) in the symplectic case was given in \cite{Takeda2025Kurokawa}.
		\item This Condition (A) corresponds to Case (I) in \cite{ibukiyama1999differential}.
		      The other Case (II) in \cite{ibukiyama1999differential} is a generalization of the Rankin-Cohen type differential operators  in \cite{Cohen1975sums},
		      and a representation-theoretic interpretation in the symplectic and unitary cases was given by Ban \cite{ban2006rankin}.
		      Rankin--Cohen-type differential operators on hermitian modular forms have been examined by Dunn \cite{Dunn2024Rankin} for the scalar-valued case,
		      and have even been explicitly constructed.
	\end{enumerate}
\end{rem}

We will consider the Howe duality for the Weil representation.
\begin{dfn}
	Let \(L_{n,\kappa}=(\bbc[M_{n,\kappa},M_{n,\kappa}])^\bfa\) be the family of the space of polynomials
	in the entries of \((n,\kappa_v)\)-matrices \(X_v=(X_{v,i,j})\) and \(Y_v=(Y_{v,i,j})\) over \(\mathbb{C}\).
	We put \(X=(X_v)_{\va}, Y=(Y_v)_{\va}\) and use the same notation as in the previous section.
	\begin{enumerate}
		\item We define the \((\mathfrak{g}_{n,\mathbb{C}},K)\)-module structure
		      \(l_{n,\kappa}\) on \(L_{n,\kappa}\) as follows:
		      \begin{align*}
			      l_{n,\kappa}(\kappa_{v,i,j})  & =\sum_{s=1}^{\kappa_v} X_{v,i,s}\frac{\partial}{\partial X_{v,j,s}}+\kappa_v\delta_{i,j}, \\
			      l_{n,\kappa}(\kappa'_{v,i,j}) & =\sum_{s=1}^{\kappa_v} Y_{v,i,s}\frac{\partial}{\partial Y_{v,j,s}},                      \\
			      l_{n,\kappa}(\pi^{+}_{v,i,j}) & =\sqrt{-1}\sum_{s=1}^{\kappa_v} X_{v,i,s}Y_{v,j,s},                                       \\
			      l_{n,\kappa}(\pi^{-}_{v,i,j}) & =\sqrt{-1}\sum_{s=1}^{\kappa_v} \frac{\partial^2}{\partial X_{v,i,s}\partial Y_{v,j,s}}
		      \end{align*}
		      on the \(v\)-th part of \(L_{n,\kappa}\),
		      and \(\mathfrak{g}_{n,v,\bbc}\) acts as 0 on the other parts.
		      Here, \(\kappa_{v,i,j}, \kappa'_{v,i,j}, \pi^{\pm}_{v,i,j}\) are defined in Section~\ref{sec:functionsonU(n,n)}.

		      For \((g_1, g_2)\in \prod_{\va}(\rmu(n)\times\rmu(n)) \cong K\) and \(f(X,Y)\in L_{n,\kappa}\),
		      we define
		      \[l_{n,\kappa}((g_1, g_2))f(X,Y)=\det(g_1)^\kappa f(^tg_1X,\,^tg_2Y).\]
		\item We define the left action of the family of unitary groups \(\rmu(\kappa)=\prod_{\va}\rmu(\kappa_v)\) on \(L_{n,\kappa}\) by
		      \[c\cdot f(X,Y)=f(Xc,Y \bar{c})\]
		      for \(c \in \rmu(\kappa)\) and \(f(X,Y) \in L_{n,\kappa}\).
	\end{enumerate}
	This representation \((l_{n,\kappa}, L_{n,\kappa})\) is well-defined, and we call it the Weil representation.

\end{dfn}

For an irreducible algebraic representation \((\lambda,V_\lambda)\) of \(\rmu(\kappa)\),
we put \(L_{n,\kappa}(\lambda)=\mathrm{Hom}_{\rmu(\kappa)}(V_\lambda,L_{n,\kappa})\)
and a \((\mathfrak{g}_{n,\mathbb{C}},K)\)-module structure
is induced on it from that of \(L_{n,\kappa}\).
We denote by \(L(\sigma)\) the unitary lowest weight \((\mathfrak{g}_{n,\mathbb{C}},K)\)-module with lowest \(K\)-type \(\sigma\).
Let \((\sigma,U_\sigma)\) be the highest weight module of \(K\) with a highest weight \(\sigma\),
and \((\lambda, V_\lambda)\) the highest weight module of \(\rmu(\kappa)\) with a highest weight \(\lambda\).
We will sometimes identify the irreducible representation of \(K\)
with the finite dimensional irreducible representation of \(\prod_{\va}(\mathrm{GL}_n(\bbc)\times\mathrm{GL}_n(\bbc))\).

The following notations are provided to write down the decomposition of \(L_{n,\kappa}\).
\begin{dfn}
	Let \(\Delta_{n,\kappa}\) be the family of pairs of Young diagrams \(D=(D_{1,v},D_{2,v})_{\va}\) such that
	the lengths \(\ell(D_{1,v})\) and \(\ell(D_{2,v})\) satisfy \(\ell(D_{1,v})\leq n\),
	\(\ell(D_{2,v})\leq n\),
	and \(\ell(D_{1,v})+\ell(D_{2,v})\leq \kappa_v\) for each \(\va\).
	We put \(\mathbbm{1}_n=(\underbrace{1,\ldots,1}_n;0,\ldots,0)_{\va}\)
	and \(\emptyset=(0,\cdots,0;0,\ldots,0)_{\va} \in \Delta_{n,\kappa}\).
	For \(D=(D_{1,v},D_{2,v})_{\va} \in \Delta_{n,\kappa}\) with \(D_{1,v}=(D^{(1)}_{1,v},\ldots,D^{(\ell(D_{1,v}))}_{1,v})\)
	and \(D_{2,v}=(D^{(1)}_{2,v},\ldots,D^{(\ell(D_{2,v}))}_{2,v})\),
	we define
	\begin{align*}
		\sigma_{n,\kappa}(D)  & =(\underbrace{D^{(1)}_{1,v}+\kappa_v,\ldots,D^{(\ell(D_{1,v}))}_{1,v}+\kappa_v,\kappa_v,\ldots,\kappa_v}_n;\
		\underbrace{D^{(1)}_{2,v},\ldots,D^{(\ell(D_{2,v}))}_{2,v},0,\ldots,0}_n)_{\va},                                                                                     \\
		\lambda_{n,\kappa}(D) & =(\underbrace{D^{(1)}_{1,v},\ldots,D^{(\ell(D_{1,v}))}_{1,v},0,\ldots,0,-D^{(1)}_{2,v},\ldots,-D^{(\ell(D_{2,v}))}_{2,v}}_{\kappa_v})_{\va}.
	\end{align*}
\end{dfn}

\begin{prop}\label{prop:howe}
	\begin{enumerate}
		\item We have \(L_{n,\kappa}(\lambda)\neq 0\) if and only if \(\lambda=\lambda_{n,\kappa}(D)\) for some
		      \(D \in \Delta_{n,\kappa}\).
		\item The lowest \(K\)-type of \(L_{n,\kappa}(\lambda_\kappa(D))\) is \(\sigma_{n,\kappa}(D)\).
		\item Under the joint action of \((\mathfrak{g}_{n,\mathbb{C}},K) \times
			      \rmu(\kappa)\), we have
		      \[L_{n,\kappa}\cong\bigoplus_{D \in \Delta_{n,\kappa}}L(\sigma_{n,\kappa}(D))\boxtimes V_{\lambda_\kappa(D)}.\]
	\end{enumerate}

\end{prop}
This proposition is a slightly modified version of the theorem proved by Kashiwara--Vergne~\cite{Kashiwara1978Segal} and Howe~\cite{howe1995perspectives}.
From this,
we obtain a correspondence between the highest weights of \(\prod_{\va}(\mathrm{GL}_n(\mathbb{C})\times \mathrm{GL}_n(\mathbb{C}))\)
and those of \(\rmu(\kappa)\), which is called Howe duality.

We fix positive integers \(n_1\geq\cdots\geq n_d\geq 1\) and set \(n=n_1+\cdots+n_d\).
We embed \(G_{(n_1)}\times \cdots \times G_{(n_d)}\)
(resp.  \(\mathfrak{g}_{n_1,\mathbb{C}}\oplus\cdots\oplus\mathfrak{g}_{n_d,\mathbb{C}}\),
\(\calk_{(n_1)}\times\cdots\times \calk_{(n_d)}\))
diagonally into \(G_{(n)}\) (resp. \(\mathfrak{g}_{n,\mathbb{C}}\), \(\calk_{(n)}\)).
We denote its image by \(G'\) (resp. \(\mathfrak{g}'_\mathbb{C}\), \(\calk'\)).
We denote by \(X_v^{(s)},Y_v^{(s)}\) the indeterminates of \(L_{n_s}\).
Then, we can easily check that the \(\mathbb{C}\)-isomorphism
\[\bigotimes_{s=1}^{d}L_{n_s,\kappa}\cong L_{n,\kappa}\]
given by \(X^{(s)}_{v,i,j}\mapsto X^{(n)}_{v,(n_1+\cdots+n_{s-1}+i),j},\ Y^{(s)}_{v,i,j}\mapsto Y^{(n)}_{v,(n_1+\cdots+n_{s-1}+i),j}\)
is an isomorphism of \((\mathfrak{g}'_\mathbb{C}, \calk')\times {\rmu(\kappa)}^d\)-modules.

\begin{dfn}
	If a polynomial \(f(X, Y) \in L_{n,\kappa}\) satisfies
	\[\qquad l_{n,\kappa}(\pi^-_{v,i,j})f=\sqrt{-1}\sum_{s=1}^{\kappa_v} \frac{\partial^2f}{\partial X_{v,i,s}\partial Y_{v,j,s}}=0 \text{  for any  } \va \text{ and } i,j\in\{1,\ldots,n\}, \]
	we say that \(f(X,Y)\) is a pluriharmonic polynomial for \(\rmu(\kappa)\).
	We denote by
	\(\ph{n,\kappa}\) the set of all pluriharmonic polynomials for \(\rmu(\kappa)\) in \(L_{n,\kappa}\).
\end{dfn}

The following proposition is stated with a slight modification of Lemma (5.3) and Theorem (6.3) in Chapter III of \cite{Kashiwara1978Segal}.

\begin{prop}
	\begin{enumerate}
		\item \(L_{n,\kappa}=L_{n,\kappa}^{\rmu(\kappa)}\cdot \ph{n,\kappa}\).
		\item \(L_{n,\kappa}^{\rmu(\kappa)}\) is the subspace \(\mathbb{C}[Z_{(n)}]\) of polynomials in the entries of a family \(Z_{(n)} =X\, ^tY\) of \((n,n)\)-matrices.
		\item Under the joint action of \(K \times \rmu(\kappa)\), we have
		      \[\ph{n,\kappa}\cong \bigoplus_{D \in \Delta_{n,\kappa}}U_{\sigma_{n,\kappa}(D)}\boxtimes V_{\lambda_\kappa(D)}.\]
	\end{enumerate}

\end{prop}
We denote by \(\ph{n,\kappa}(D)\) the subspace of \(\ph{n,\kappa}\)
corresponding to \(U_{\sigma_{n,\kappa}(D)}\boxtimes V_{\lambda_\kappa(D)}\) under this isomorphism.

\begin{lem}\label{lem:pluriharmonic}
	For \(D_s\in\Delta_{n,\kappa}\), we have
	\[\mathrm{Hom}_{(\mathfrak{g}'_\mathbb{C}, \calk')}
		\left( \underset{s=1}{\overset{d}{\otimes}} L(\sigma_{n_s,\kappa}(D_s)), L(\kappa\mathbbm{1}_n)\right)
		\cong \left(\left(\underset{s=1}{\overset{d}{\otimes}}\ph{n_s,\kappa}(D_s)\right)^{\rmu(\kappa)}
		\otimes \left(\underset{s=1}{\overset{d}{\otimes}}U_{\sigma'_{n_s,\kappa}(D_s)}\right)\right)^{\calk'}.
	\]
	Here, \(\sigma'_{n_s,\kappa}(D_s)\) is the representation defined by
	\(\sigma'_{n_s,\kappa}(D_s)(k)=\sigma_{n_s,\kappa}(D_s)({}^tk^{-1})\) for \(k \in \calk_{(n_s)}\).
\end{lem}
\begin{proof}
	We note that \(L_{n,\kappa}(\emptyset)=L_{n,\kappa}^{\rmu(\kappa)} \cong L(\kappa\mathbbm{1}_n)\) by Proposition~\ref{prop:howe}.
	We have
	\begin{eqnarray*}\mathrm{Hom}_{(\mathfrak{g}'_\mathbb{C},\calk')}
		\left(\underset{s=1}{\overset{d}{\otimes}} L(\sigma_{n_s,\kappa}(D_s)), L_{n,\kappa}\right)
		&\cong&  \underset{s=1}{\overset{d}{\otimes}}\mathrm{Hom}_{(\mathfrak{g}_{n_s,\mathbb{C}}, \calk_{(n_s)})}
		\left(L(\sigma_{n_s,\kappa}(D_s)),L_{n_s,\kappa}\right)\\
		&\cong&  \underset{s=1}{\overset{d}{\otimes}}\mathrm{Hom}_{\calk_{(n_s)}}
		\left( U_{\sigma_{n_s,\kappa}(D_s)}, \ph{n_s,\kappa}\right)\\
		&=&  \underset{s=1}{\overset{d}{\otimes}}\mathrm{Hom}_{\calk_{(n_s)}}
		\left( U_{\sigma_{n_s,\kappa}(D_s)}, \ph{n_s,\kappa}(D_s)\right)\\
		&\cong& \underset{s=1}{\overset{d}{\otimes}}\left(\ph{n_s,\kappa}(D_s)
		\otimes U_{\sigma'_{n_s,\kappa}(D_s)}\right)^{\calk_{(n_s)}}\\
		&\cong& \left(\left(\underset{s=1}{\overset{d}{\otimes}}\ph{n_s,\kappa}(D_s)\right)
		\otimes \left(\underset{s=1}{\overset{d}{\otimes}}U_{\sigma'_{n_s,\kappa}(D_s)}\right)\right)^{\calk'}.
	\end{eqnarray*}
	Restricting to the \(\rmu(\kappa)\)-invariant subspace gives the desired isomorphism.
\end{proof}

There is a natural injection
\begin{align*}
	\left(\underset{s=1}{\overset{d}{\otimes}}\ph{n_s,\kappa}(D_s)\right)^{\rmu(\kappa)}
	\otimes \left(\underset{s=1}{\overset{d}{\otimes}}U_{\sigma'_{n_s,\kappa}(D_s)}\right)
	 & \hookrightarrow \left(\underset{s=1}{\overset{d}{\otimes}}L_{n_s,\kappa}(D_s)\right)^{\rmu(\kappa)}
	\otimes \left(\underset{s=1}{\overset{d}{\otimes}}U_{\sigma'_{n_s,\kappa}(D_s)}\right)                                               \\
	 & \hookrightarrow L_{n,\kappa}^{\rmu(\kappa)}\otimes \left(\underset{s=1}{\overset{d}{\otimes}}U_{\sigma'_{n_s,\kappa}(D_s)}\right) \\
	 & \cong \mathbb{C}[Z_{(n)}]\otimes \left(\underset{s=1}{\overset{d}{\otimes}}U_{\sigma'_{n_s,\kappa}(D_s)}\right).
\end{align*}
We denote the image of \(h \in \left(\underset{s=1}{\overset{d}{\otimes}}\ph{n_s,\kappa}(D_s)\right)^{\rmu(\kappa)}
	\otimes \left(\underset{s=1}{\overset{d}{\otimes}}U_{\sigma'_{n_s,\kappa}(D_s)}\right)\) by \(\Phi_h(Z_{(n)})\).

Let \(\Gamma_{n}\) be a discrete subgroup of \(G_{n}\).
Note that
\[
	\mathrm{Hom}_{K}\!\left(U_\sigma,\, C_\mathrm{mod}^\infty(\Gamma_n\backslash G)\right)
	\;\cong\;
	\Bigl[C_\mathrm{mod}^\infty(\Gamma_n\backslash G)\otimes U_{\sigma'}\Bigr]^{K},
\]
where \(\sigma'\) is the representation defined by
\(\sigma'(k)=\sigma({}^tk^{-1})\) for \(k \in K\),
which is isomorphic to the contragredient representation \(\sigma^*\) of \(\sigma\).
From this, we obtain the following well-known isomorphism.

\begin{prop}\label{prop:holohom}
	We have the isomorphism
	\[\mathrm{Hom}_{(\mathfrak{g}_{n,\mathbb{C}}, K)}\left(L(\sigma),C_\mathrm{mod}^\infty(\Gamma_n\backslash G)\right)
		\cong\left[C_\mathrm{mod}^\infty(\Gamma_n\backslash G)\otimes U_{\sigma'}\right]^{K,\mathfrak{p}_n^-=0}.\]
\end{prop}

Under this isomorphism,
we denote by \(I_F\in\mathrm{Hom}_{(\mathfrak{g}_{n,\mathbb{C}},K)}\left(L(\sigma),C_\mathrm{mod}^\infty(\Gamma_n\backslash G)\right)\)
the corresponding homomorphism of \(F \in \left[C_\mathrm{mod}^\infty(\Gamma_n\backslash G)\otimes U_{\sigma'}\right]^{\calk_{n,\infty},\mathfrak{p}_n^-=0}\).

We take the discrete subgroup \(\Gamma_{n_s}\) of \(G_{n_s}\) for \(s=1,\ldots,d\).
Let \(\Gamma'\) be the image of \(\Gamma_{n_1}\times\cdots\times\Gamma_{n_d}\) in \(G_n\).

\begin{thm}\label{thm:maintype}
	Let \(F\) be a hermitian modular form of type \(\kappa\mathbbm{1}_{n}\) for \(\Gamma'\)
	and take \(D_s \in \Delta_{n_s,\kappa}\) for \(s=1,\ldots,d\).
	We put \(\pi_{n}^+=(\pi^+_{v,i,j})_{\va}\in (M_{n}(\mathfrak{p}_n^+))^\bfa\).
	We denote by \(\mathrm{Res}\) the pullback of the functions on \(G_{(n)}\)
	by the diagonal embedding \(G_{(n_1)}\times\cdots\times G_{(n_d)} \hookrightarrow G_{(n)}\).
	Then, we have
	\[\mathrm{Res}\left(\Phi_{h}(\pi_{n}^+)F\right)
		\in\bigotimes_{s=1}^d\left[C_\mathrm{mod}^\infty(\Gamma_{n_s}\backslash {G_{(n_s)}})\otimes U_{\sigma'_{n_s,\kappa}(D_s)}\right]^{\calk_{(n_s)},\mathfrak{p}_{n_s}^-=0}\]
	for any  \(h \in \left(\left(\underset{s=1}{\overset{d}{\otimes}}\ph{n_s,\kappa}(D_s)\right)^{\rmu(\kappa)}
		\otimes \left(\underset{s=1}{\overset{d}{\otimes}}U_{\sigma'_{n_s,\kappa}(D_s)}\right)\right)^{\calk'}\).
\end{thm}
\begin{proof}
	This can be proved in exactly the same way as Theorem~4.10 in \cite{Takeda2025Kurokawa}.
\end{proof}

Using Proposition~\ref{prop:automisom},
we translate Theorem~\ref{thm:maintype} into a theorem on hermitian modular forms on the hermitian upper space \(\hus{n}\).

\begin{thm}\label{thm:mainweight}
	Let \(F\) be a hermitian modular form in \(M_\kappa(\Gamma_K^{(n)})\)
	and take \(D_s \in \Delta_{n_d,\kappa}\) for \(s=1,\ldots,d\).
	Then, we have
	\[\mathrm{Res}\left((\Phi_{h}(\partial_{Z})) F\right)\in \underset{s=1}{\overset{d}{\otimes}} M_{(\sigma_{n_s,\kappa}(D_s))}(\Gamma_K^{(n_s)})\]
	for any \(h \in\) {\small \(\left(\left(\underset{s=1}{\overset{d}{\otimes}}\ph{n_s,\kappa}(D_s)\right)^{\rmu(\kappa)}
		\otimes \left(\underset{s=1}{\overset{d}{\otimes}}U_{\sigma'_{n_s,\kappa}(D_s)}\right)\right)^{\calk'}\)}.
\end{thm}
Before the proof, we provide some notation and a lemma.

\begin{dfn}
	For a holomorphic function \(f\) on \(\hus{n}^\bfa\) and a representation \((\sigma, U_\sigma)\) of \(K^\bbc:=\prod_{\va}(\mathrm{GL}_n(\mathbb{C})\times \mathrm{GL}_n(\mathbb{C}))\),
	we define the function \(\tilde{f}\) on \(G\)
	and  the representation \((\widetilde{\sigma}, U_\sigma)\) of \(G\) as follows:
	\[\tilde{f}(g)=f(g\left<\bfi_n\right>),\qquad \tilde{\sigma}(g)=\sigma(j(g,\bfi_n))\]
	for \(g\in G\).
\end{dfn}

We take the section of \(G\ni g\mapsto g\left<\bfi_n\right>\in\hus{n}^\bfa\) by
\[Z=X+\sqrt{-1}Y\in\hus{n} \mapsto g_Z:=\begin{pmatrix} Y^{1/2} & XY^{-1/2}\\ 0& Y^{-1/2}\\ \end{pmatrix}\in G.\]

\begin{lem}[c.f. Ban \cite{ban2006rankin}]\label{lem:pideri}
	For a holomorphic function \(f\) on \(\hus{n}^\bfa\) and a representation \((\sigma, U_\sigma)\) of \(K^\bbc\),
	we have
	\begin{eqnarray*}
		(\pi_{v,i,j}^+\tilde{f})(g)&=&2(\mu(g)^{-1}\cdot\widetilde{\partial_Zf}(g)\cdot{}^t\lambda(g)^{-1})_{v,i,j},\\
		(\pi_{v,i,j}^+\tilde{\sigma})(g)&=&\sqrt{-1}\tilde{\sigma}(g)\cdot d\sigma(\lambda(g)^{-1}\cdot\overline{\mu(g)}\cdot e_{v,i,j},0).
	\end{eqnarray*}

	In particular, for \(Z=X+\sqrt{-1}Y\), we have
	\begin{eqnarray*}
		(\pi_{v,i,j}^+\tilde{f})(g_Z)&=&2({}^tY^{1/2}\cdot\widetilde{\partial_Zf}(g_Z)\cdot {}^tY^{1/2})_{v,i,j},\\
		(\pi_{v,i,j}^+\tilde{\sigma})(g)&=&\sqrt{-1}\tilde{\sigma}(g)\cdot d\sigma(e_{v,i,j},0).
	\end{eqnarray*}
\end{lem}

Note that the definition of \(M(g,Z)\) in this paper is different from \(j(g,Z)\) in \cite{ban2006rankin}.

For \(D \in \Delta_{n,\kappa}\),
we denote by \(\rho_{n,D}\) the representation of \(K^\bbc\) with a dominant integral weight \(D\).
(Then, \(\sigma_{n,\kappa}(D)=\det^{\kappa}\otimes\rho_{n,D}\).)

\begin{proof}[Proof of Theorem~\ref{thm:mainweight}]
	From Theorem~\ref{thm:maintype},
	\[
		\mathrm{Res}\left(\Phi_{h}(\pi_{n}^+)\cdot \phi_F\right) \in
		\bigotimes_{s=1}^d\left[C_\mathrm{mod}^\infty(\Gamma_K^{(n_s)}\backslash G_{n_s})\otimes
		U_{\sigma'_{n_s,\kappa}(D_s)}\right]^{\calk_{(n_s)},\mathfrak{p}_{n_s}^-=0}.
	\]
	By applying Proposition~\ref{prop:automisom} to each factor in the tensor product,
	there exists \(f \in \bigotimes_{s=1}^d M_{\sigma_{n_s,\kappa}(D_s)}(\Gamma_K^{(n_s)})\)
	such that \(\phi_f=\mathrm{Res}\left(\Phi_{h}(\pi_{n}^+)\cdot \phi_F\right)\).
	Since
	\[\phi_f(g_{Z_1},\ldots,g_{Z_d})=(\underset{s=1}{\overset{d}{\otimes}}\det(Y_s)^{\kappa/2}(\rho_{n_s,D_s}(Y_s^{1/2},{}^tY_s^{1/2})))f(Z_1,\ldots,Z_d),\]
	we have
	\[
		f(Z_1,\ldots,Z_d)=(\underset{s=1}{\overset{d}{\otimes}}
		\det(Y_s^{1/2})^{-\kappa}\rho_{n_s,D_s}(Y_s^{1/2},{}^tY_s^{1/2})^{-1})
		\mathrm{Res}\left(\Phi_{h}(\pi_{n}^+)\cdot \phi_F\right)(g_{Z_1},\ldots,g_{Z_d}).
	\]
	We now consider the action of \(\pi_{v,i,j}^+\) on \(\phi_F\).
	Note that \(\phi_F(g)=(\widetilde{\det^\kappa}\cdot\widetilde{F})(g)\) using the notation above.

	Under the isomorphism \(\hus{n}^\bfa\times K\cong G\),
	we regard \(\phi_F\) as a function of \(Z=X+\sqrt{-1}Y\in \hus{n}^\bfa\),
	\(Y^{1/2}\),
	and \(k\in K\).

	From Lemma~\ref{lem:deri}, we can easily check that
	the highest degree part of \(\Phi_{h}(\pi_{n}^+)\cdot \phi_F\) in \(Y^{1/2}\)
	is \(2^{m_h}\det(Y)^{\kappa/2} \cdot \Phi_h({}^tY^{1/2}\partial_Z{}^tY^{1/2})\cdot\phi_F\),
	where \(m_h\) is a degree of \(\Phi_h\).

	Since \(h \in\) {\small \(\left(\left(\underset{s=1}{\overset{d}{\otimes}}\ph{n_s,\kappa}(D_s)\right)^{\rmu(\kappa)}
		\otimes \left(\underset{s=1}{\overset{d}{\otimes}}U_{\sigma'_{n_s,\kappa}(D_s)}\right)\right)^{\calk'}\)},
	we have
	\[
		\mathrm{Res}(2^{m_h}\det(Y)^{\kappa/2} \cdot \Phi_h({}^tY^{1/2}\partial_Z{}^tY^{1/2})\cdot\phi_F)
		=2^{m_h}(\underset{s=1}{\overset{d}{\otimes}}\det(Y_s)^{\kappa/2}(\rho_{n_s,D_s}(Y_s^{1/2},{}^tY_s^{1/2})))\mathrm{Res}( \Phi_h(\partial_Z)\cdot\phi_F).
	\]
	Thus, we may denote \(\mathrm{Res}\left(\Phi_{h}(\pi_{n}^+)\cdot \phi_F\right)\) by
	\begin{align*}
		\mathrm{Res}\left(\Phi_{h}(\pi_{n}^+)\cdot \phi_F\right)(g_1,\ldots,g_d)
		=2^{m_h} & (\underset{s=1}{\overset{d}{\otimes}}\det(Y_s)^{\kappa/2}(\rho_{n_s,D_s}(Y_s^{1/2},{}^tY_s^{1/2})))\mathrm{Res}( \Phi_h(\partial_Z)\cdot\phi_F) \\
		         & +(\underset{s=1}{\overset{d}{\otimes}}\det(Y_s)^{\kappa/2})R,
	\end{align*}
	where \(R=R(Z_1,\ldots,Z_d,Y_1^{1/2},\ldots,Y_d^{1/2},k_1,\ldots,k_d)\) is
	a polynomial with a degree strictly lower than
	that of \(\underset{s=1}{\overset{d}{\otimes}}\rho_{n_s,D_s}(Y_s^{1/2},{}^tY_s^{1/2})\) in \(({Y_1}^{1/2},\ldots,{Y_d}^{1/2})\).
	Then we have
	\[f=2^{m_h}\mathrm{Res}( \Phi_h(\partial_Z)\cdot\phi_F)+(\underset{s=1}{\overset{d}{\otimes}}\rho_{n_s,D_s}(Y_s^{1/2},{}^tY_s^{1/2})^{-1})R.\]
	On the other hand, since \(f\) is a holomorphic function, we have \(R=0\).
	Therefore, \(\mathrm{Res}( \Phi_h(\partial_Z)\cdot\phi_F)=2^{-m_h}f\)
	is an element of \(\underset{s=1}{\overset{d}{\otimes}} M_{(\sigma_{n_s,\kappa}(D_s))}(\Gamma_K^{(n_s)})\).
\end{proof}

In particular,
it can be rewritten in analogy with the results of Ibukiyama~\cite{ibukiyama1999differential} as follows:

\begin{cor}\label{cor:diff}
	Let \(n_1,\ldots,n_d\) be positive integers such that \(n_1\geq\cdots\geq n_d\geq 1\)
	and put \(n = n_1+\cdots+n_d\).
	We take a family \((\bfk_s, \bfl_s)=(\bfk_{v,s},\bfl_{v,s})_{\va}\) of pairs of dominant integral weights such that \(\ell(\bfk_{v,s})\leq n_d\), \(\ell(\bfl_{v,s})\leq n_d\)
	and \(\ell(\bfk_{v,s})+\ell(\bfl_{v,s})\leq\kappa_v\) for each \(\va\) and \(s=1,\ldots,d\).

	Let \(P_v(T)\) be a \(\left(V_{n_1,\bfk_{v,1},\bfl_{v,1}}\otimes\cdots\otimes V_{n_d,\bfk_{v,d},\bfl_{v,d}}\right)\)-valued polynomial on a space of degree \(n\) variable matrices \(M_n\) for \(\va\),
	and put \(P(T)=(P_v(T))_{\va}\).
	The differential operator \(\mathbb{D}=P(\partial_{Z})=(P_v(\partial_{Z_v}))_{\va}\) satisfies Condition (A) for \(\det^\kappa\)
	and \(\det^\kappa\rho_{n_1,\bfk_{1},\bfl_1}\otimes\cdots\otimes\det^\kappa\rho_{n_d,\bfk_d,\bfl_d}\) if and only if
	\(P(T)\) satisfies the following conditions:
	\begin{enumerate}
		\item If we put \(\widetilde{P}(X_1,\ldots,X_d,Y_1,\ldots,Y_d)=P\left(
			      \begin{pmatrix}
					      X_1\,^tY_1 & \cdots & X_1\,^tY_d \\
					      \vdots     & \ddots & \vdots     \\
					      X_d\,^tY_1 & \cdots & X_d\,^tY_d
				      \end{pmatrix}\right)\) with \(X_i, Y_i \in (M_{n_i,\kappa_v})_{\va}\),
		      then \(\widetilde{P}\) is pluriharmonic for each \((X_i, Y_i)\).
		\item For \((A_i,B_i) \in \calk_{(n_i)}^\bbc:=\prod_{\va}(\mathrm{GL}_{n_i}(\bbc)\times\mathrm{GL}_{n_i}(\bbc))\),
		      we have
		      \[P\left(
			      \begin{pmatrix} A_1 & & \\&\ddots&\\ & &  A_d\end{pmatrix}
			      T
			      \begin{pmatrix} ^t\!B_1 & & \\&\ddots&\\ & &  ^t\!B_d\end{pmatrix}
			      \right)
			      =\left(\rho_{n_1,\bfk_1,\bfl_1}(A_1,B_1)\cdots\otimes\rho_{n_d,\bfk_d,\bfl_d}(A_d,B_d)\right)P(T).\]
	\end{enumerate}
\end{cor}
The above theorems and corollary do not say anything about the existence of differential operators satisfying condition (A) or how to construct them.
In general, even finding the dimension of the space formed by these differential operators is a difficult problem.
However,
when \(d=2\),
it is easy to determine when such a differential operator exists.

\begin{prop}
	The notations are the same as in the above corollary.
	When \(d=2\),
	there exists a differential operator \(\mathbb{D}\) satisfying condition (A) for \(\det^\kappa\) and \(\det^\kappa\rho_{n_1,\bfk_{1},\bfl_1}\otimes\det^\kappa\rho_{n_2,\bfk_{2},\bfl_2}\) if and only if \(\bfk_1=\bfl_2\) and \(\bfl_1=\bfk_2\).
	If it exists,
	it is unique up to scalar multiplication.
\end{prop}
\begin{proof}
	It follows from the fact that
	\[
		\dim_\bbc{\small \left(\left(\underset{s=1}{\overset{2}{\otimes}}\ph{n_s,\kappa}(D_s)\right)^{\rmu(\kappa)}
		\otimes \left(\underset{s=1}{\overset{2}{\otimes}}U_{\sigma'_{n_s,\kappa}(D_s)}\right)\right)^{\calk'}}
		=\dim_\bbc (V_{\lambda_\kappa(D_1)}\otimes V_{\lambda_\kappa(D_2)})^{\rmu(\kappa)}\]
	is equal to	1 if  \(V_{\lambda_\kappa(D_1)}\) and \(V_{\lambda_\kappa(D_2)}\) are contragredient representations of each other and 0 otherwise.
\end{proof}

\section{Hermitian Eisenstein Series}\label{sec:eisenstein}
In this section, we introduce the hermitian Eisenstein series according to Shimura
\cite[\S 16.5]{shimura2000arithmeticity}. We fix a family of positive integers \(\kappa=(\kappa_v)_{\va }\) and an integral ideal \(\frakn\) of \(\kp\).

Define the following subgroups of \(G_{n}\) for \(r\leq n\) by
\begin{align*}
	L_{n,r} & =\left\{\left.
	\begin{pmatrix}
		A & 0   & 0            & 0   \\
		0 & I_r & 0            & 0   \\
		0 & 0   & \adj{A}^{-1} & 0   \\
		0 & 0   & 0            & I_r \\\end{pmatrix}
	\in G_{n}\right| A \in \mathrm{GL}_{n-r}(K)\right\}, \\
	U_{n,r} & =\left\{
	\begin{pmatrix}
		I_{n-r} & *   & *       & *   \\
		0       & I_r & *       & 0   \\
		0       & 0   & I_{n-r} & 0   \\
		0       & 0   & *       & I_r \\\end{pmatrix}
	\in G_{n}\right\},                                   \\
	G_{n,r} & =\left\{
	\begin{pmatrix}
		I_{n-r} & 0 & 0       & 0 \\
		0       & * & 0       & * \\
		0       & 0 & I_{n-r} & 0 \\
		0       & * & 0       & * \\\end{pmatrix}
	\in G_{n}\right\}.
\end{align*}
Then the subgroups \(P_{n,r}=G_{n,r}L_{n,r}U_{n,r}\) are the standard parabolic subgroups of \(G_n\) and there are natural embeddings
\(t_{n,r}: \mathrm{GL}_{n-r}(K)\hookrightarrow L_{n,r}\) and \(s_{n,r}: G_{r}\hookrightarrow G_{n,r}\).
Define \(G_{n,r,v}, G_{n,r, \adele}, L_{n,r,v}, L_{n,r, \adele}\), etc. in the same way as \(G_{n,v}, G_{n, \adele}\), etc.

By the Iwasawa decomposition, \(G_{n, \adele}\) (resp. \(G_{n,v}\)) can be decomposed as
\(G_{n, \adele}=P_{n,r,\adele}\calk_{n,\adele}\) (resp. \(G_{n,v}=P_{n,r,v}\calk_{n,v}\)).

We take a Hecke character \(\chi\) of \(K\) with infinite part \(\chi_\infty\) given by
\[
	\chi_\infty(x_\infty)
	= \prod_{v \in \Sigma_K} |x_v|^{-\kappa_v} x_v^{\kappa_v},
\]
and with conductor dividing \(\frakn\).
We put \(\chi_v=\prod_{w|v}\chi_w\) for a finite place \(v\) of \(\kp\).

\begin{dfn}
	We define
	\[\epsilon_{n,\kappa,v}(g,s; \frakn, \chi)=\left\{
		\begin{array}{ll}
			\norm{\det(\adj{A}A)}_v^{s}\chi_v(\det A^*)         & (v\in\bfh \text{ and } k\in \calk_{n,v}(\frakn) ),     \\
			0                                                   & (v\in\bfh \text{ and } k\not\in \calk_{n,v}(\frakn) ), \\
			\norm{\delta(g)}^{\kappa_v-2s}\delta(g)^{-\kappa_v} & (v\in\bfa)
		\end{array}
		\right. \]
	for a complex variable \(s\) and \(g=t_{n,0}(A)\mu k \in G_{n,v}\)
	with \(A\in \mathrm{GL}_n(K_v)\), \(\mu\in U_{n,0}\) and \(k\in \calk_{n,v}\).
	Then, we put
	\[\epsilon_{n,\kappa}(g,s;\frakn,\chi)=\prod_v\epsilon_{n,\kappa,v}(g_v,s;\frakn,\chi),\]
	and define the hermitian Eisenstein series \(E_{n,\kappa}(g,s;\frakn,\chi)\) on \(G_{n, \adele}\) by
	\[E_{n,\kappa}(g,s;\frakn,\chi)=\sum_{\gamma\in P_{n,0}\backslash \rmu_n(\kp)}\epsilon_{n,\kappa}(\gamma g,s; \frakn, \chi).\]
	For \(\theta\in \calk_{n,0}\), we put
	\[E_{n,\kappa}^\theta(g,s;\frakn,\chi)=E_{n,\kappa}(g\theta^{-1},s; \frakn,\chi).\]
\end{dfn}
The hermitian Eisenstein series \(E_{n,\kappa}(g,s;\frakn,\chi)\) and \(E^\theta_{n,\kappa}(g,s;\frakn,\chi)\) converge absolutely and locally uniformly for \(\Re(s)>n\)
(see, for example, \cite{Shimura1997Euler}).
\begin{prop}[{\cite[Proposition~17.7.]{shimura2000arithmeticity}}]
	Let \(\mu\) be a positive integer such that \(\mu\geq n\).
	If \(\kappa_v=\mu\) for any \(\va\), then \(E_{n,\kappa}(g,\mu/2;\frakn,\chi)\) belongs to \(\mathcal{A}_n(\det^\mu,\frakn)\)
	except when \(\mu=n+1\), \(\kp=\bbq\), \(\chi=\chi_K^{n+1}\),
	where \(\chi_K\) is the quadratic character associated to quadratic extension \(K/\kp\).
\end{prop}

Let \(\rho_{r}\) be the representation of \(\calk_{(r)}^\bbc\) with a family \((\bfk,\bfl)=(k_{1,v},\ldots,k_{r,v};l_{1,v},\ldots,l_{r,v})_{\va}\) of dominant integral weights. For \(n\geq r\), we define the representation \(\rho_{n}\) of \(\calk_{(n)}^\bbc\) as the representation corresponding to a family \((\bfk',\bfl')=(k_{1,v},\ldots,k_{r,v},k_{r,v},\ldots,k_{r,v}; l_{1,v},\ldots,l_{r,v},0,\ldots,0)_{\va}\) of dominant integral weights.

\begin{dfn}
	We define
	\[\epsilon(f)_{r,v}^n(g,s ;\chi)=\left\{
		\begin{array}{ll}
			\abs{\det \left(\adj{A_r}A_r\right)}_v^{s}\chi_v(\det A_r^*)f(h_r)                  & (v\in\bfh \text{ and } k\in \calk_{n,v}(\frakn)  ),    \\
			0                                                                                   & (v\in\bfh \text{ and } k\not\in \calk_{n,v}(\frakn) ), \\
			\norm{\delta(g)\delta(h_r)^{-1}}^{\kappa_v-2s}\rho_n(M(g))^{-1}\rho_r(M(h_r))f(h_r) & (v\in\bfa)
		\end{array}
		\right.\]
	for \(f\in \mathcal{A}_{0,r}(\rho_r,\frakn)\) (\(r<n\))
	and \(g=t_{n,r}(A_r)\, \mu_r\,	s_{n,r}(h_r)\, k \in G_{n,v}\) with \(A_r\in \mathrm{GL}_{n-r}(K_v)\), \(\mu_r\in U_{n,r}\),
	\(h_r\in G_{r,v}\) and \(k\in \calk_{n,v}\).
	Then, we put
	\[\epsilon(f)_{r}^n(g,s;\chi)=\prod_v\epsilon(f)_{r,v}^n(g_v,s ;\chi)\]
	and define the hermitian Eisenstein series \(\left[f\right]_{r}^n(g,s;\chi)\) on
	\(G_{n, \adele}\) associated with \(f\) by
	\[\left[f\right]_r^n(g,s;\chi)=\sum_{\gamma\in P_{n,r}\backslash \rmu_n(\kp)}\epsilon(f)_{r}^n(\gamma g,s;\chi).\]
\end{dfn}

\section{Pullback Formula}
We fix positive integers \(n_1, n_2\) such that \(n_1\geq n_2\),
an integral ideal \(\frakn\) of \(\kp\),
a family \(\kappa=(\kappa_v)_{\va}\) of positive integers such that \(\kappa_v\geq n_1+n_2\) for any \(\va\),
and a family \((\bfk, \bfl)=(\kv,\lv)_{\va }\)
of pairs of dominant integral weights such that \(\ell(\bfk_v)\leq n_2\), \(\ell(\bfl_v)\leq n_2\)
and \(\ell(\bfk_v)+\ell(\bfl_v)\leq\kappa_v\) for each \(\va\).
Put
\(n=n_1+n_2\) and \(G_{n_1,n_2}=G_{n_1}\times G_{n_2}\).
We set \(\rho_r:=\det^\kappa \rho_{r,(\bfk, \bfl)}\) and \(\rho'_r:=\det^\kappa \rho_{r,(\bfl, \bfk)}\) for a positive integer \(r\geq\max\left\{\ell(\bfk_v),\ell(\bfl_v)\mid \va\right\}\).

\subsection{Double Coset Decomposition}\label{sec:decomp}
We define a natural injection \(\iota\) by
\begin{eqnarray*}
	\iota:G_{n_1}\times G_{n_2}&\rightarrow& G_{n}\\
	\left(\begin{pmatrix}A_1 & B_1 \\ C_1 & D_1 \\ \end{pmatrix}, \begin{pmatrix}A_2 & B_2 \\ C_2 & D_2 \\ \end{pmatrix}\right)&\mapsto&
	\begin{pmatrix}A_1 &0 & B_1&0 \\ 0&A_2&0&B_2\\C_1 &0 & D_1&0 \\ 0&C_2&0&D_2\\ \end{pmatrix}.
\end{eqnarray*}

We put \(g^\natural=\begin{pmatrix}0&I_r\\I_r&0\\ \end{pmatrix}
	g\begin{pmatrix}0&I_r\\I_r&0\\ \end{pmatrix}\) and \(f^\natural(g)=f(g^\natural)\)
for \(g \in G \) (\(G= G_r, G_{r,\adele}, G_{r,v},G_{r,\infty},\ldots\)).
To describe the main theorem, we define \(f^\dagger_\chi\) by
\[
	f^\dagger_\chi(g)
	= \chi(\det g)\,\overline{f^\natural(g)},
\]
where \(\chi\) is a Hecke character of \(K\) whose archimedean component
\(\chi_\infty\) satisfies the same condition as in
Section~\ref{sec:eisenstein}.
For each place \(v\), we define \(f^\dagger_{\chi,v}\) to be the corresponding local component of \(f^\dagger_\chi\).

The following are well known facts (see, for example, \cite[\S 3]{Garrett1984Pullbacks} and \cite[Appendix B]{mizumoto1991poles}).
\begin{fact}\label{fact:decomp}
	\begin{enumerate}
		\item The double coset \(P_{n,0}\backslash G_{n}/G_{n_1,n_2}\) has an irredundant set of
		      representatives
		      \[\left\{\left.\xi_r=\begin{pmatrix}I_{n_1}               & 0               & 0       & 0       \\
               0                     & I_{n_2}         & 0       & 0       \\
               0                     & \widetilde{I_r} & I_{n_1} & 0       \\
               {}^{t}\widetilde{I_r} & 0               & 0       & I_{n_2} \\\end{pmatrix}\right|0\leq r\leq n_2\right\},\]
		      where \(\widetilde{I_r}=\begin{pmatrix}0&0\\0&I_{r}\\ \end{pmatrix}\in M_{n_1,n_2}(\mathbb{Z})\).
		\item \(P_{n,0} \, \xi_r \, \iota(g_1,g_2)=P_{n,0}\xi_r\) if and only if
		      \(g_1=h_1\,s_{n_1,r}(g)\) and \(g_2=h_2\,s_{n_2,r}(g^\natural)\)
		      with \(h_1\in L_{n_1,r}U_{n_1,r}\), \(h_2\in L_{n_2,r}U_{n_2,r}\) and \(g\in G_{r}\).
		      In particular,  \(P_{n,0}\backslash P_{n,0}\,\xi_r G_{n_1,n_2}\) has an irredundant set of coset representatives
		      \[\left\{\xi_r\iota(\gamma_1,\gamma_2)\mid \gamma_1\in P_{n_1,r}\backslash G_{n_1},
			      \gamma_2\in L_{n_2,r}U_{n_2,r}\backslash G_{n_2}\right\}.\]
	\end{enumerate}
\end{fact}

\subsection{Pullback Formula}
From now on, we assume \(n_1=n_2\) if \(\frakn \neq \inte{\kp}\).
We fix an element \(\theta=(\theta_v)\in \calk_{n,0}\) as
\[\theta_v=\left\{
	\begin{array}{ll}
		I_{2n}    & (v\nmid \frakn), \\
		\xi_{n_2} & (v\mid \frakn).
	\end{array}
	\right.\]

Let \(\chi_{K}\) be the quadratic character associated to quadratic extension \(K/\kp\).
For a Hecke eigenform \(f\in \mathcal{A}_{0,\rho_{n}}(\rmu_{n})\)
and a Hecke character \(\eta\) of \(K\), we set
\[D(s,f; \eta)
	=L(s-n+1/2,f\otimes \eta,\mathrm{St})
	\cdot\left(\prod_{i=0}^{2n-1} L_{\kp}(2s-i,\eta\cdot\chi_{K}^{i})\right)^{-1},
\]
where \(L(*,f\otimes {\eta}, \mathrm{St})\) is the standard L-function attached
to \(f\otimes {\eta}\) and \(L_{\kp}(*,\eta)\) (resp. \(L_{\kp}(*,\eta\cdot\chi_{K})\)) is
the Hecke L-function attached to \(\eta\) (resp. \(\eta\cdot\chi_{K}\)).
For a finite set \(S\) of finite places \(v\), we put
\[D_S(s,f;\eta)=\prod_{\substack{v\in\bfh- S}}D_v(s,f;\eta),\]
where \(D_v(s,f;\eta)\) is the \(v\)-component of \(D(s,f; \eta)\).

In this section, we prove the pullback formula.
Let \(\dkl\) be the differential operator satisfying the condition (A) for \(\det^\kappa\) and \(\rho'_{n_1}\otimes\rho_{n_2}\).
We fix a Hecke eigenform \(f=\prod_v f_v\in \mathcal{A}_{n_2}(\rho_{n_2},\frakn)\).

From Fact~\ref{fact:decomp} and the definition of the hermitian Eisenstein series, we
have
\[
	(\dkl E^\theta_{n,\kappa})(\iota(g_1,g_2),s; \frakn,  \chi)
	=\sum_{r=0}^{n_2}
	\sum_{\gamma_1\in P_{n_1,r}\backslash G_{n_1}}
	\sum_{\gamma_2\in P_{n_2,r}\backslash G_{n_2}}
	W^\theta_r(\gamma_1g_1,\gamma_2g_2, s; \frakn,  \chi),
\]
where
\[W^\theta_r(g_1,g_2 ,s; \frakn,  \chi)=\sum_{\gamma_2'\in G_{n_2,r}}
	(\dkl\epsilon_{n,\kappa})\left(\xi_r\,\iota(g_1,\gamma_2'g_2)\theta^{-1}  ,s; \frakn,  \chi\right)\]
for \((g_1,g_2)\in G_{n_1,\adele}\times G_{n_2,\adele}\).

\begin{prop}
	For any Hecke eigenform \(f\in \mathcal{A}_{0,\rho_{n_2}}(\rmu_{n_2})\) and \(r<n_2\), we have
	\[	\int_{ G_{n_2}\backslash G_{n_2,\adele}}\left<f(g_2) , \sum_{\gamma_2\in P_{n_2,r}\backslash G_{n_2}}W^\theta_{r}(g_1,\gamma_2g_2 ,s; \frakn, \chi)\right>dg_2=0.\]
\end{prop}
\begin{proof}
	We have
	\begin{align*}
		 & \int_{ G_{n_2}\backslash G_{n_2,\adele}}\left<f(g_2) , \sum_{\gamma_2\in P_{n_2,r}\backslash G_{n_2}}W^\theta_{r}(g_1,\gamma_2g_2 ,s; \frakn, \chi)\right>dg_2               \\
		 & \quad=	\int_{ L_{n_2,r}U_{n_2,r}\backslash G_{n_2,\adele}}\left<f(g_2) ,(\dkl\epsilon_{n,\kappa})\left(\xi_r\,\iota(g_1,g_2)\theta^{-1}  ,s; \frakn,  \chi\right)\right>dg_2 \\
		 & \quad=	\int_{ L_{n_2,r}U_{n_2,r,\adele}\backslash G_{n_2,\adele}}\int_{U_{n_2,r}\backslash U_{n_2,r,\adele}}
		\left<f(ug) ,(\dkl\epsilon_{n,\kappa})\left(\xi_r\,\iota(g_1,ug)\theta^{-1}  ,s; \frakn,  \chi\right)\right>dudg.
	\end{align*}
	Since
	\[
		\xi_r\,\iota(I_{n_1},u)\,\xi_r^{-1} =
		\begin{pmatrix}
			I_{n_1}                                                & 0 & 0       & 0 \\[1mm]
			\begin{pmatrix} 0 & * \\ 0 & 0 \end{pmatrix}           &
			\begin{pmatrix} I_{n_2-r} & * \\ 0 & I_r \end{pmatrix} & 0 &
			\begin{pmatrix} * & * \\ * & 0 \end{pmatrix}                             \\[1mm]
			0                                                      & 0 & I_{n_1} &
			\begin{pmatrix} 0 & 0 \\ * & 0 \end{pmatrix}                             \\[1mm]
			0                                                      & 0 & 0       &
			\begin{pmatrix} I_{n_2-r} & 0 \\ * & I_r \end{pmatrix}
		\end{pmatrix}
		\in P_{n,0,\adele}
	\]
	for any \(u\in U_{n_2,r,\adele}\),
	we have
	\[
		\epsilon_{n,\kappa}\Bigl(\xi_r\,\iota(I_{n_1},u)\,\xi_r^{-1},\, s; \frakn, \chi\Bigr) = 1
		\quad \text{and} \quad
		\Delta\bigl(\xi_r\,\iota(I_{n_1},u)\,\xi_r^{-1}\bigr) = \bfi_n.
	\]
	Hence,
	\[
		(\dkl\epsilon_{n,\kappa})\Bigl(\xi_r\,\iota(g_1, u g)\,\theta^{-1},\, s; \frakn, \chi\Bigr)
		= c \cdot (\dkl\epsilon_{n,\kappa})\Bigl(\xi_r\,\iota(g_1, g)\,\theta^{-1},\, s; \frakn, \chi\Bigr),
	\]
	with some constant \(c \in \mathbb{C}\) which depends on \(\dkl\) but does not depend on \(u \in U_{n_2,r,\mathbb{A}}\).

	Therefore, the last integral is equal to
	\[\overline{c}\cdot\int_{ L_{n_2,r}U_{n_2,r,\adele}\backslash G_{n_2,\adele}}
		\left<\int_{U_{n_2,r}\backslash U_{n_2,r,\adele}}f(ug)du ,(\dkl\epsilon_{n,\kappa})\left(\xi_r\,\iota(g_1,g)\theta^{-1}  ,s; \frakn,  \chi\right)\right>dg.\]
	Since \(f\) is a cusp form and \(U_{n_2,r}\) is a unipotent radical of a proper parabolic subgroup,
	this integral is equal to 0.
\end{proof}
\begin{rem}
	These computations are (essentially) known. See, for example \cite[Satz 2]{Klingen1967Zum}.
\end{rem}
By this proposition, we consider only the case \(r=n_2\).
We have
\begin{align*}
	(f ,W^\theta_{n_2}(g_1,*,\overline{s} ; \frakn,  \chi))
	 & =\int_{ G_{n_2}\backslash G_{n_2,\adele}}\left<f(g_2),W^\theta_{n_2}(g_1,g_2,\overline{s} ; \frakn,  \chi)\right>dg_2                                  \\
	 & =\int_{G_{n_2,\adele}}\left<f(g_2),(\dkl\epsilon_{n,\kappa})\left(\xi_{n_2}\,\iota(g_1,g_2)\theta^{-1},\overline{s} ; \frakn,  \chi\right)\right>dg_2.
\end{align*}
Therefore, the last integral can be decomposed into a product of local factors.

\subsection{Local Computations}
\subsubsection{Good Non-archimedean Factors}
Let \(v\) be a finite place of \(\kp\) that corresponds to a prime ideal \(\mathfrak{p}\) of \(\kp\) such that \(\mathfrak{p}\nmid\frakn\).
For \(g_{1,v}=t_{n_1,n_2}(A_1)\mu_1 s_{n_1,n_2}(h) k_1 \in G_{n_1,v}\)
with \(A_1\in \mathrm{GL}_{n_1-n_2}(K_v)\), \(\mu_1\in U_{n_1,n_2,v}\), \(h\in G_{n_2,v}\) and \(k_1\in \calk_{n_1,v}\),
we have
\begin{align*}
	  & \quad\int_{G_{n_2,v}}f_v(g_2)\,\overline{\epsilon_{n,\kappa,v}\left(\xi_{n_2}\,\iota(g_{1,v},g_2),\overline{s}; \frakn,  \chi\right)}\,dg_2 \\
	= & \abs{\det \left(\adj{A_1}A_1\right)}_v^{s}\overline{\chi_v}(\det A_1^*)
	\int_{G_{n_2,v}}f_v(g_2)\,\overline{\epsilon_{n,\kappa,v}\left(\xi_{n_2}\,\iota(s_{n_1,n_2}(h),g_2),\overline{s}; \frakn, \chi\right)}\,dg_2    \\
	= & \abs{\det \left(\adj{A_1}A_1\right)}_v^{s}\overline{\chi_v}(\det A_1^*h)
	\int_{G_{n_2,v}}f_v^\natural(h{g}^{-1})\,\chi_v(\det (g))\overline{\epsilon_{n,\kappa,v}\left(\xi_{n_2}\,\iota(s_{n_1,n_2}(g),I_{n_2}),\overline{s}; \frakn, \chi\right)}\,dg.
\end{align*}
We put \(\eta(g)=\chi_v(\det(g))\overline{\epsilon_{n,\kappa,v}\left(\xi_{n_2}\,\iota(s_{n_1,n_2}(g),I_{n_2}),\overline{s}; \frakn, \chi_v\right)}\).
Since
\begin{align*}
	\eta(kgk')
	 & = \chi_v(\det(kgk'))\,
	\overline{\epsilon_{n,\kappa,v}\!\left(
		\xi_{n_2}\,\iota\bigl(s_{n_1,n_2}(kgk'), I_{n_2}\bigr),
		\overline{s}; \frakn, \chi_v
	\right)}                  \\
	 & = \chi_v(\det(gk'))\,
	\overline{\epsilon_{n,\kappa,v}\!\left(
		\xi_{n_2}\,\iota\bigl(s_{n_1,n_2}(gk'), {k^\natural}^{-1}\bigr),
		\overline{s}; \frakn, \chi_v
	\right)}                  \\
	 & = \eta(g),
\end{align*}
for any \(k, k' \in G_{n_2}(\mathcal{O}_v)\) (noting that \(\chi_v\) is unramified),
it follows that \(\eta(g)\) is left and right \(G_{n_2}(\mathcal{O}_v)\)-invariant.

So, it can be written as a limit of elements of the Hecke algebra \(\mathcal{H}_v(G_{n_2,v},G_{n_2}(\mathcal{O}_v))\),
consisting of compactly supported functions on \(G_{n_2,v}\) that are bi-invariant under \(G_{n_2}(\mathcal{O}_v)\).

Since \(f_v\) is an \(v\)-eigenfunction,
there is a constant \(S_v(f_v)\) such that
\[\int_{G_{n_2,v}}f_v^\natural(h{g}^{-1})\,\eta(g)\,dg=S_v(f_v)f_v^\natural(h).\]

We review the Satake homomorphism \cite{Satake1963Theory} to determine the constant \(S_v(f_v)\).

\vskip.5\baselineskip

Let \(G\) be a reductive linear algebraic group over \(\mathfrak{p}\)-adic field \(F_\mathfrak{p}\)
and take a maximal open compact subgroup \(\calk\).
We denote by \(\mathcal{H}_\mathfrak{p}(G,\calk)\) the Hecke algebra of a pair \((G,\calk)\).
Let \(T\) be a maximal \(F_\mathfrak{p}\) split torus in \(G\), \(M\) the centralizer of \(T\) in \(G\),
\(B\) a minimal parabolic subgroup of \(G\) containing \(M\),
and \(U\) the unipotent radical of \(B\).
Let \(du\) and \(dm\) be the left Haar measures on U and M, respectively,
normalized so that the volumes of \(U\cap \calk\) and \(M\cap \calk\) are 1.
Let \(\delta_M: M\rightarrow \bbr^\times\) be the modular function on \(M\).
Let \(W_T:=N_T/M\) where \(N_T\) is the normalizer of \(T\) in \(G\) be the Weyl group of \(T\) in \(G\).
\(W_T\) acts on \(\mathcal{H}_\mathfrak{p}(M,M\cap \calk)\) as
\[w\cdot f(m)=f^w(m)=f(wmw^{-1})\]
for \(w \in W_T\) and \(f \in \mathcal{H}_\mathfrak{p}(G,\calk)\).
\begin{prop}[\cite{Satake1963Theory}]
	The map \(S_U:\mathcal{H}_\mathfrak{p}(G,\calk)\rightarrow\mathcal{H}_\mathfrak{p}(M,M\cap \calk)^{W_T}\) given by
	\[S_Uf(m)=\delta_M^{-\frac{1}{2}}(m)\int_Uf(um)du=\delta_M^{\frac{1}{2}}(m)\int_Uf(mu)du\]
	is an algebra isomorphism.
\end{prop}

\vskip.5\baselineskip

Returning to our setting, we assume that \(v\) is non-split (i.e., \(\mathfrak{p}\) is inert or ramified).

As the centralizer \(M\) of maximal \(\kp_v\)-split torus, we take
\[M=\left\{t=\begin{pmatrix}
		t_1 &        &         &                     &        &                         \\
		    & \ddots &         &                     &        &                         \\
		    &        & t_{n_2} &                     &        &                         \\
		    &        &         & \overline{t_1}^{-1} &        &                         \\
		    &        &         &                     & \ddots &                         \\
		    &        &         &                     &        & \overline{t_{n_2}}^{-1} \\
	\end{pmatrix}\right\}.\]
Let \(B\) be the minimal parabolic subgroup of \(G\) consisting of upper triangular
matrices and \(N\) the unipotent radical of \(B\). In this case, we have the
natural identification \(\mathcal{H}_\mathfrak{p}(M,M\cap
	\calk_{n_2,v})\cong\bbc[T_1,T_1^{-1},\cdots,T_{n_2},T_{n_2}^{-1}]\).

Let \(\varpi_v\) be a prime element of \(K_v\).
The double coset \(\calk_{n_2,v}\backslash
	G_{n_2,v}/\calk_{n_2,v}\) has an irredundant set of representatives
\[\left\{\left.\varpi_{d_1, \ldots, d_{n_2}}=\begin{pmatrix}
		\varpi_v^{d_1} &        &                    &                            &        &                                \\
		               & \ddots &                    &                            &        &                                \\
		               &        & \varpi_v^{d_{n_2}} &                            &        &                                \\
		               &        &                    & \overline{\varpi_v}^{-d_1} &        &                                \\
		               &        &                    &                            & \ddots &                                \\
		               &        &                    &                            &        & \overline{\varpi_v}^{-d_{n_2}} \\
	\end{pmatrix}\right|d_1\geq\cdots\geq d_{n_2}\right\}.\]

We put
\[\varpi_+=\mathrm{diag}({\varpi_v^{d_1}, \ldots,\varpi_v^{d_k}, 1, \ldots,1})\]
and \[\varpi_-=\mathrm{diag}(1,\ldots,1,\varpi_v^{d_{k+1}},	\ldots,\varpi_v^{d_{n_2}})\]
for an element \(\varpi=\varpi_{d_1, \ldots, d_{n_2}} \ (d_1\geq\cdots\geq	d_k\geq 0 > d_{k+1}\geq\cdots\geq d_{n_2})\) of the set.
Then, we have
\begin{align*}
	  & \xi_{n_2}\,  \iota(s_{n_1,n_2}(\varpi),I_{n_2}) \\
	= & {\small
			\begin{pmatrix}
				I_{t} & 0        & 0                        & 0     & 0                        & 0                \\
				0     & \varpi_+ & 0                        & 0     & 0                        & \varpi_+\varpi_- \\
				0     & 0        & \overline{\varpi_-}^{-1} & 0     & I_{n_2}                  & 0                \\
				0     & 0        & 0                        & I_{t} & 0                        & 0                \\
				0     & 0        & 0                        & 0     & \overline{\varpi_+}^{-1} & 0                \\
				0     & 0        & 0                        & 0     & 0                        & \varpi_-
			\end{pmatrix}
			\begin{pmatrix}
				I_{t} & 0                    & 0                                          & 0     & 0                        & 0             \\
				0     & \varpi_-(1-\varpi_+) & 0                                          & 0     & 0                        & -I_{n_2}      \\
				0     & 0                    & \overline{\varpi_-}(1-\overline{\varpi_+}) & 0     & -I_{n_2}                 & 0             \\
				0     & 0                    & 0                                          & I_{t} & 0                        & 0             \\
				0     & 0                    & \overline{\varpi_+}                        & 0     & \overline{\varpi_-}^{-1} & 0             \\
				0     & \varpi_+             & 0                                          & 0     & 0                        & \varpi_-^{-1}
			\end{pmatrix}}
\end{align*}
where \(t=n_1-n_2\),
and the right matrix of the product is an element of \(\calk_{n_2,v}\).
Thus, we have
\[\eta(\varpi)=\abs{\det(\varpi^2_+\varpi_-^{-2})}_v^{s}\overline{\chi_v}(\det(\varpi_+\varpi_-^{-1})).\]
Here we note that \(\eta\) is left and right \(G_{n_2}(\mathcal{O}_v)\)-invariant.
When \(v\) is split, \(G_{n_2,v}\) is isomorphic to \(\mathrm{GL}_{2n_2}(\kp_v)\). If
\(M\), \(B\), and \(N\) are defined similarly, we have the natural
identification \(\mathcal{H}_\mathfrak{p}(M,M\cap
	\calk_{n_2,v})\cong\bbc[T_1,T_1^{-1},\cdots,T_{2n_2},T_{2n_2}^{-1}]\),
and a similar calculation can be made as in the non-split case.

By \cite[Theorem~19.8]{shimura2000arithmeticity},
the following
holds.
\begin{thm} Under the Satake isomorphism, we have
	\[S_N\eta=\left\{
		\begin{array}{ll}
			\dfrac{\prod_{i=1}^{2n_2}(1-(-1)^{i-1}q^{i-1-2s}\overline{\chi_v}(\mathfrak{p}))}
			{\prod_{i=1}^{n_2}(1-q^{2n_2-2s-2}\,\overline{\chi_v}(\mathfrak{P})T_i)(1-q^{2n_2-2s}\,\overline{\chi_v}(\mathfrak{P})T_i^{-1})}
			 & (\text{inert; } \mathfrak{p}=\mathfrak{P}),                 \\[15pt]
			\dfrac{\prod_{i=0}^{n_2-1}(1-q^{2i-2s}\overline{\chi_v}(\mathfrak{p}))}
			{\prod_{i=1}^{n_2}(1-q^{n_2-s-1}\,\overline{\chi_v}(\mathfrak{P})T_i)(1-q^{n_2-s}\,\overline{\chi_v}(\mathfrak{P})T_i^{-1})}
			 & (\text{ramified; } \mathfrak{p}=\mathfrak{P}^2),            \\[15pt]
			\displaystyle\prod_{i=1}^{2n_2}\dfrac{1-q^{i-1-2s}\,\overline{\chi_v}(\mathfrak{p})}
			{(1-q^{2n_2-s}\,\overline{\chi_{v}}(\mathfrak{P}_1)T_i^{-1})(1-q^{-1-s}\,\overline{\chi_{v}}(\mathfrak{P}_2)T_i)}
			 & (\text{split; } \mathfrak{p}=\mathfrak{P}_1\mathfrak{P}_2),
		\end{array}
		\right.
	\]
	where \(q=N_{\kp/\bbq}(\mathfrak{p})\).
\end{thm}

Let \(\lambda_{j}(f_v)\) be the eigenvalue of \(f_v\) for the Hecke operator \(T_j\).
Then, \(T_jf_v^\natural=\lambda_{j}(f_v)f_v^\natural\). Therefore, we have the
specific formula for \(S_v(f_v)\).
\begin{cor} We have
	\begin{align*}
		S_v(f_v) & =\left\{
		\begin{array}{ll}
			\dfrac{\prod_{i=1}^{2n_2}(1-(-1)^{i-1}q^{i-1-2s}\overline{\chi_v}(\mathfrak{p}))}
			{\prod_{i=1}^{n_2}(1-q^{2n_2-2s-2}\lambda_{i}(f_v)\,\overline{\chi_v}(\mathfrak{P}))
			(1-q^{2n_2-2s}\lambda_{i}(f_v)^{-1}\,\overline{\chi_v}(\mathfrak{P}))}
			 & (\text{inert; } \mathfrak{p}=\mathfrak{P}),                 \\[15pt]
			\dfrac{\prod_{i=0}^{n_2-1}(1-q^{2i-2s}\overline{\chi_v}(\mathfrak{p}))}
			{\prod_{i=1}^{n_2}(1-q^{n_2-s-1}\lambda_{i}(f_v)\,\overline{\chi_v}(\mathfrak{P}))(1-q^{n_2-s}\lambda_{i}(f_v)^{-1}\,\overline{\chi_v}(\mathfrak{P}))}
			 & (\text{ramified; } \mathfrak{p}=\mathfrak{P}^2),            \\[15pt]
			\displaystyle\prod_{i=1}^{2n_2}\dfrac{1-q^{i-1-2s}\,\overline{\chi_v}(\mathfrak{p})}
			{(1-q^{2n_2-s}\lambda_{i}(f_v)^{-1}\,\overline{\chi_{v}}(\mathfrak{P}_1))(1-q^{-1-s}\lambda_{i}(f_v)\,\overline{\chi_{v}}(\mathfrak{P}_2))}
			 & (\text{split; } \mathfrak{p}=\mathfrak{P}_1\mathfrak{P}_2),
		\end{array}
		\right.
		\\
		         & =L_v(s-n_2+1/2,f\otimes \overline{\chi},\mathrm{St})
		\cdot\left(\prod_{i=0}^{2n_2-1} L_{\kp}(2s-i,\overline{\chi}\cdot\chi_{K}^{i})\right)^{-1} \\
		         & =D_v(s,f;\overline{\chi}).
	\end{align*}
\end{cor}

Therefore, we obtain the following proposition.
\begin{prop}\label{prop:good} For a finite place \(v\in\bfh\) such that \(v\nmid\frakn\), we have
	\[\int_{G_{n_2,v}}f_v(g_2)\,\overline{\epsilon_{n,\kappa,v}\left(\xi_{n_2}\,\iota(g_1,g_2),\overline{s}; \frakn, \chi\right)}\,dg_2
		=D_v(s,f;\overline{\chi})\overline{\epsilon(f^\dagger_{\chi,v})_{n_2,v}^{n_1}(g_1,s;\chi)}.\]
\end{prop}

\subsubsection{Bad Non-archimedean Factors}
Let \(v\in\bfh\) be a finite place of \(\kp\)  such that \(v\mid\frakn\).
For simplicity, we may consider only the case \(n_1 = n_2\).
We have
\begin{align*}
	  & \int_{G_{n_2,v}}f_v(g_2)\,\overline{\epsilon_{n,\kappa,v}\left(\xi_{n_2}\,\iota(g_{1,v},g_2)\theta^{-1},\overline{s}; \frakn, \chi\right)}\,dg_2                                                                \\
	= & \overline{\chi_v}(\det(g_{1,v}))\int_{G_{n_2,v}}f_v^\natural(g_{1,v}{g}^{-1})\chi_v(\det(g))\overline{\epsilon_{n,\kappa,v}\left(\xi_{n_2}\,\iota(g,I_{n_2})\theta^{-1},\overline{s}; \frakn, \chi\right)}\,dg.
\end{align*}

As in \cite[p.462]{Garrett1992On}, we choose an explicit integral representation
for \(\epsilon_{n,\kappa,v}\) in order to facilitate explicit computations.
Let \(\phi_v\) be the characteristic function of
\(\left\{\begin{pmatrix}u&v\end{pmatrix}\in M_{n,2n}(\inte{K_v})\mid \begin{pmatrix}u&v\end{pmatrix}\equiv \begin{pmatrix}0_n&I_n\end{pmatrix} \pmod{\frakn \mathcal{O}_{K_v}}\right\}\)
on \(M_{n,2n}(K_v)\).
Then we have
\[\epsilon_{n,\kappa,v}\left(g,s; \frakn, \chi\right)
	=\mathrm{vol}(\calk_{n,v}(\frakn))^{-1}\int_{\mathrm{GL}_n(K_v)}\norm{\det(\adj{t}t)}_v^{s}\chi_v(\det t)\phi_v(t\begin{pmatrix}0_n&I_n\end{pmatrix}g)dt,\]
where \(dt\) is a Haar measure on \(\mathrm{GL}_n(K_v)\) such that \(\mathrm{GL}_n(\inte{K_v})\) has volume 1,
and \(\mathrm{vol}(\calk_{n,v}(\frakn))\) is the measure of \(\calk_{n,v}(\frakn)\) with respect to the Haar measure on \(G_{n,v}\).

From the definition of \(\phi_v\),
\(\phi_v(t\begin{pmatrix}0_n&I_n\end{pmatrix}\xi_{n_2}\,\iota(g,I_{2n_2})\theta^{-1})\neq0\) if and only if
\[t\begin{pmatrix}0_n&I_n\end{pmatrix}\xi_{n_2}\,\iota(g,I_{2n_2})
	\equiv \begin{pmatrix}
		0_{n_2} & I_{n_2} & I_n & 0_{n_2} \\I_{n_2}&0_{n_2}&0_{n_2}&I_{n_2}
	\end{pmatrix} \pmod{\frakn \mathcal{O}_{K_v}}.\]
By a simple calculation, we have \(\phi_v(t\begin{pmatrix}0_n&I_n\end{pmatrix}\xi_{n_2}\,\iota(g,I_{2n_2})\theta^{-1})\neq0\)
if and only if \(t\in \mathrm{GL}_n(\inte{K_v})\), \(t\equiv I_n  \pmod{\frakn \mathcal{O}_{K_v}} \),
and \(g\equiv I_n \pmod{\frakn \mathcal{O}_{K_v}} \).
Thus, we have
\[\overline{\chi_v}(\det(g))\epsilon_{n,\kappa,v}\left(\xi_{n_2}\,\iota(g,I_{2n_2})\theta^{-1},s; \frakn, \chi\right)=\left\{\begin{array}{ll}
		1 \quad & (g\in \calk_{n,v}(\frakn)),     \\
		0 \quad & (g\not\in \calk_{n,v}(\frakn)), \\
	\end{array}\right.\]
since the conductor of \(\chi\) divides \(\frakn\).
Therefore, since \(f^\natural\) is right \(\calk_{n,v}(\frakn)\)-invariant, we obtain the following proposition.
\begin{prop}\label{prop:bad}
	We assume \(n_1=n_2\). For a finite place \(v\) of \(\kp\) such that \(v\mid\frakn\), we have
	\[\int_{G_{n_2,v}}f_v(g_2)\,\overline{\epsilon_{n,\kappa,v}\left(\xi_{n_2}\,\iota(g_{1,v},g_2),\overline{s}; \frakn, \chi\right)}\,dg_2=[\calk_{n,v}:\calk_{n,v}(\frakn)]\overline{f^\dagger_{\chi,v}(g_{1,v})}.\]
\end{prop}

\subsubsection{Archimedean Factors}
We put \(\epsilon_{n,\kappa,\infty}(g,s; \frakn,\chi)=\prod_{\va}\epsilon_{n,\kappa,v}(g_v,s; \frakn,\chi)\)
for \(g=(g_v)_{\va}\in G_{n,\infty}\).
From Lemma~\ref{lem:difpoly}, there is the family \(Q(T,s)=(Q_v(T,s))_{\va}\) of polynomials such that
\[\dkl\epsilon_{n,\kappa,\infty}(g,s; \frakn,\chi)=\epsilon_{n,\kappa,\infty}(g,s)Q(\Delta(g),s).\]
By Corollary~\ref{cor:acttophi} and Corollary~\ref{cor:diff}, we can easily obtain the
following lemma.
\begin{lem}
	We have
	\[Q\left(
		\begin{pmatrix} A_1 & 0 \\ 0 &  A_2\end{pmatrix}
		T
		\begin{pmatrix}^t\!B_1 & 0\\ 0 & ^t\!B_2  \end{pmatrix},s
		\right)
		=\left(\rho_{n_1,\bfl,\bfk}(A_1,B_1)\otimes\rho_{n_2,\bfk,\bfl}(A_2,B_2)\right)Q(T,s)\]
	for \((A_i,B_i)\in \calk_{(n_i)}^\bbc=\prod_{\va}(\mathrm{GL}_{n_i}(\bbc)\times\mathrm{GL}_{n_i}(\bbc))\).
\end{lem}

We put \(|\bfk|=\sum_{v,i}k_{v,i}\), \(|\bfl|=\sum_{v,i}l_{v,i}\), \(|\kappa|=n_2\sum_{\va}\kappa_v\) and \(\abs{\rho_{n_2}}=|\kappa|+|\bfk|+|\bfl|\)
for the fixed dominant integral weights such that \(\bfk_v=(k_{v,1},k_{v,2},\ldots)\), \(\bfl_v=(l_{v,1},l_{v,2},\ldots)\) for each \(\va\).

In the following,
the subscript \(\infty\) is often omitted,
and the notation of Section~3 will be used.

\begin{prop}\label{prop:arch}
	We have
	\begin{equation*}
		\int_{G_{n_2,\infty}}\left<f(g_2),(\dkl\epsilon_{n,\kappa,\infty})\left(\xi_{n_2}\,\iota(g_1,g_2),\overline{s}; \frakn, \chi\right)\right>dg_2
		=c(s,\rho_{n_2})\cdot\overline{\epsilon(f^\dagger_\chi)_{n_2}^{n_1}(g_1,s; \chi)}.
	\end{equation*}
	Here, for any \(w\in V_{\rho_{n_2}}\), the function \(c(s,\rho_{n_2})\) satisfies
	\begin{align*}
		c(s,\rho_{n_2})w=2^{-mn_2(2s+2n_2)+2|\kappa|+|\bfk|+|\bfl|} & \int_{\mathfrak{S}_{n_2}}\left<\rho_{n_2}(I_{n_2}-\adj{S}S,I_{n_2}-{}^{t}\!{S}\overline{S})w,
		Q(R,\overline{s})\right>                                                                                                                                    \\
		                                                            & \qquad\cdot\det(I_{n_2}-\adj{S}S)^{\kappa/2-s-2n_2}dS,
	\end{align*}
	where \(\mathfrak{S}_{n_2}=\left\{S\in
		(M_{n_2}(\bbc))^\bfa\mid I_{n_2}-\adj{S}S>0\right\}\) and
	\begin{align*}
		R= &
		\begin{pmatrix}
			\begin{pmatrix}0&0\\0&\sqrt{-1}S\end{pmatrix}
			 & \begin{pmatrix}0\\I_{n_2}\end{pmatrix}   \\
			\begin{pmatrix}0&I_{n_2}\end{pmatrix}
			 & -\sqrt{-1}\adj{S}(I_{n_2}-S\adj{S})^{-1}
		\end{pmatrix}.
	\end{align*}
\end{prop}

\begin{proof}
	We set \(Z=\begin{pmatrix}Z_{11} & Z_{21}  \\ Z_{22} & Z'\end{pmatrix}
		=g_1\left<\bfi_{n_1}\right>\ (Z'\in \hus{n_2}^\bfa)\),
	\(W=g_2\left<\bfi_{n_2}\right>\in \hus{n_2}^\bfa\).
	We put \(Y_1=\Im(Z)\), \(Y'_1=\Im(Z')\) and \(Y_2=\Im(W)\).

	Since
	\begin{align*}
		\Delta(\xi_{n_2}\,\iota(g_1,g_2))
		= & \begin{pmatrix}\lambda(g_1)&0\\0&\lambda(g_2)\end{pmatrix}^{-1}
		\begin{pmatrix}
			I_{n_1}-\widetilde{I_{n_2}}W{}^{t}\!\widetilde{I_{n_2}}Z & 0                                                        \\
			0                                                        & I_{n_2}-{}^{t}\!\widetilde{I_{n_2}}Z\widetilde{I_{n_2}}W
		\end{pmatrix}^{-1}                                                                                                    \\
		  & \quad \cdot\begin{pmatrix}
			               \sqrt{-1}(I_{n_1}-\widetilde{I_{n_2}}W^{t}\!\widetilde{I_{n_2}}\adj{Z})Y_1^{-1} & 2\widetilde{I_{n_2}}                                                              \\
			               2 {}^{t}\!\widetilde{I_{n_2}}                                                   & \sqrt{-1}(I_{n_2}-{}^{t}\!\widetilde{I_{n_2}}Z\widetilde{I_{n_2}}\adj{W})Y_2^{-1}
		               \end{pmatrix} \\
		  & \quad\cdot\begin{pmatrix}{}^{t}\!\mu(g_1)&0\\0&{}^{t}\!\mu(g_2)\end{pmatrix}^{-1},
	\end{align*}
	we have
	\begin{align*}
		\quad & \dkl\epsilon_{n,\kappa,\infty}(\xi_{n_2}\,\iota(g_1,g_2),s; \frakn, \chi)                              \\
		      & \quad= \abs{\delta(g_1)\delta(g_2)\det(I_{n_1}-\widetilde{I_{n_2}}g_2\left<\bfi_{n_2}\right>
		{}^{t}\!\widetilde{I_{n_2}}g_1\left<\bfi_{n_1}\right>)}^{\kappa-2s}                                            \\
		      & \quad\qquad\cdot\left(\delta(g_1)\delta(g_2)\det(I_{n_1}-\widetilde{I_{n_2}}g_2\left<\bfi_{n_2}\right>
		{}^{t}\!\widetilde{I_{n_2}}g_1\left<\bfi_{n_1}\right>)\right)^{-\kappa}
		\cdot Q\left(\Delta(\xi_{n_2}\,\iota(g_1,g_2)),s\right)                                                        \\
		      & \quad = \abs{\delta(g_1)\delta(g_2)\det(I_{n_2}-WZ')}^{\kappa-2s}
		\cdot \det(I_{n_2}-WZ')^{-\kappa}                                                                              \\
		      & \quad\qquad\cdot\rho'_{n_1}(M(g_1))^{-1}\otimes\rho_{n_2}(M(g_2))^{-1}Q(R_1,s),
	\end{align*}
	where
	\footnotesize
	\begin{align*}
		R_1= &
		\begin{pmatrix}
			\sqrt{-1}\begin{pmatrix}I_{n_1-n_2}&0\\-WZ_{21}&I_{n_2}-WZ'\end{pmatrix}^{-1}
			\begin{pmatrix}I_{n_1-n_2}&0\\-W\adj{Z_{12}}&I_{n_2}-W\adj{Z'}\end{pmatrix}Y_1^{-1}
			 & \begin{pmatrix}0\\2(I_{n_2}-WZ')^{-1}\end{pmatrix}     \\
			\begin{pmatrix}0&2(I_{n_2}-Z'W)^{-1}\end{pmatrix}
			 & \sqrt{-1}(I_{n_2}-Z'W)^{-1}(I_{n_2}-Z'\adj{W})Y_2^{-1}
		\end{pmatrix}.
	\end{align*}
	\normalsize
	By left translation with \(w_0=\begin{pmatrix}0       & -I_{n_2} \\
               I_{n_2} & 0        \\
		\end{pmatrix}\),
	we have
	\begin{align}
		 & \int_{G_{(n_2)}}\left<f(g_2),(\dkl\epsilon_{n,\kappa,\infty})\left(\xi_{n_2}\,\iota(g_1,g_2),\overline{s}\right)\right>dg_2 \label{eq:int} \\
		 & \quad =  \overline{\rho'_{n_1}}(M(g_1))^{-1}\int_{G_{(n_2)}}\abs{\delta(g_1)\delta(g_2)\det(I_{n_2}-WZ')}^{\kappa-2s}
		\det(I_{n_2}-\overline{W}\overline{Z'})^{-\kappa}  \notag                                                                                     \\
		 & \qquad\cdot\left<f(g_2),\rho_{n_2}(M(g_2))^{-1}\: Q(R_1,\overline{s})\right>dg_2  \notag                                                   \\
		 & \quad =  \overline{\rho'_{n_1}}(M(g_1))^{-1}\int_{G_{(n_2)}}\abs{\delta(g_1)\delta(g_2)\det(Z'+W)}^{\kappa-2s}
		\det(\overline{Z'}+\overline{W})^{-\kappa}\notag                                                                                              \\
		 & \qquad\cdot\left<f(w_0g_2),\rho_{n_2}(M(g_2))^{-1}\: Q(R_2,\overline{s})\right>dg_2 \notag                                                 \\
		 & \quad  = \overline{\rho'_{n_1}}(M(g_1))^{-1}\int_{\hus{n_2}^\bfa}\abs{\delta(g_1)\det(Y_2)^{-1/2}\det(Z'+W)}^{\kappa-2s}
		\det(\overline{Z'}+\overline{W})^{-\kappa}\notag                                                                                              \\
		 & \qquad\cdot\left<\rho_{n_2}(Y_2^{1/2}W^{-1},\,^{t}Y_2^{1/2}\,{}^tW^{-1})F(-W^{-1}),\rho_{n_2}(Y_2^{1/2},\,^{t}Y_2^{1/2})
		Q(R_2,\overline{s})\right>\frac{dW}{\det(Y_2)^{2n_2}}
		\notag                                                                                                                                        \\
		 & \quad = \overline{\rho'_{n_1}}(M(g_1))^{-1}\int_{\hus{n_2}^\bfa}\abs{\delta(g_1)\det(Y_2)^{-1/2}\det(Z'+W)}^{\kappa-2s}
		\det(\overline{Z'}+\overline{W})^{-\kappa}
		\notag                                                                                                                                        \\
		 & \qquad\cdot\left<\rho_{n_2}(Y_2W^{-1},\,^{t}Y_2\,{}^tW^{-1})F(-W^{-1}), Q(R_2,\overline{s})\right>\frac{dW}{\det(Y_2)^{2n_2}},  \notag
	\end{align}
	where
	\footnotesize
	\begin{align*}
		R_2= &
		\begin{pmatrix}
			\sqrt{-1}\begin{pmatrix}I_{n_1-n_2}&0\\-(Z'+W)^{-1}(Z_{21}-\adj{Z_{12}})&(Z'+W)^{-1}(W+\adj{Z'})\end{pmatrix}Y_1^{-1}
			 & \begin{pmatrix}0\\2(Z'+W)^{-1}\end{pmatrix} \\
			\begin{pmatrix}0&2(Z'+W)^{-1}\end{pmatrix}
			 & \sqrt{-1}(Z'+W)^{-1}(Z'+\adj{W})Y_2^{-1}
		\end{pmatrix}
	\end{align*}
	\normalsize
	and \(F(W)=\rho_{n_2}(M(g_2))f(g_2)\in M_{\rho_{n_2}}(\Gamma_K^{(n_2)})\) for \(W=g_2\left<\bfi_{n_2}\right>\in\hus{n_2}^\bfa\).

	By the Cholesky decomposition, a positive definite Hermitian matrix can be
	written as the product of a lower triangular matrix and its conjugate transpose.
	In our setting, this implies that there exists  a matrix \(F_0=\begin{pmatrix}F_1 & F_2\\ 0 &F_3\end{pmatrix} \in (\mathrm{GL}_n(\bbc))^\bfa\)
	such that \(Y_1^{-1}=\adj{F_0}\,F_0\) and \({Y'_1}^{-1}=\adj{F_3}\,F_3\).
	We set
	\[S=\mathcal{L}_{Z'}(W)=F_3(\adj{Z'}+W)(Z'+W)^{-1}F_3^{-1}.\]
	Then, the map \(\mathcal{L}_{Z'_1}: \hus{n_2}^\bfa\rightarrow\left\{S\in
		(M_{n_2}(\bbc))^\bfa\mid I_{n_2}-\adj{S}S>0\right\}=:\mathfrak{S}_{n_2}\) is biholomorphic. We
	note
	\begin{align*}
		dW  & =2^{-2n_2^2m}\det(Y_2)^{2n_2}\abs{\det((I_{n_2}-\adj{S}S))}^{-2n_2}dS, \\
		Y_2 & =2^{-2}(\,\adj{Z'}+\adj{W})\,\adj{F_3}(I_{n_2}-\adj{S}S)F_3(Z'+W)      \\
		    & =2^{-2}(Z'+W)\,\adj{F_3}(I_{n_2}-S\adj{S})F_3(\,\adj{Z'}+\adj{W}),
	\end{align*}
	where \(dS\) is defined in the same way as \(dW\).
	We put
	\[\hat{F}(S):=\rho_{n_2}\left((Z'+W)W^{-1}, ({}^t{Z'}+ {}^t{W}){}^t{W}^{-1}\right)F(-W^{-1}).\]
	Then the integral \eqref{eq:int} is equal to
	\begin{align}
		 & 2^{|\kappa|-mn_2(2s+2n_2)-2\abs{\rho_{n_2}}} \abs{\delta(g_1)\det({Y'_1})^{1/2}}^{\kappa-2s}\overline{\rho'_{n_1}}(M(g_1))^{-1}\notag \\
		 & \quad \cdot\int_{\mathfrak{S}_{n_2}} \left<\hat{F}(S),
		\rho_{n_2}(\,\adj{F_3}(I_{n_2}-\adj{S}S)F_3, \,^t\!F_3(I_{n_2}-\overline{S}{}^{t}\!{S})\overline{F_3})Q(R_3,\overline{s})\right>		\notag \\
		 & \quad \qquad\cdot\det(I_{n_2}-\adj{S}S)^{\kappa/2-s-2n_2}dS, \label{eq:int2}
	\end{align}
	where
	\footnotesize
	\begin{align*}
		R_3= &
		\begin{pmatrix}
			\sqrt{-1}\begin{pmatrix}I_{n_1-n_2}&0\\-(Z'+W)^{-1}(Z_{21}-\adj{Z_{12}})&(Z'+W)^{-1}(W+\adj{Z'})\end{pmatrix}Y_1^{-1}
			 & \begin{pmatrix}0\\2I_{n_2}\end{pmatrix} \\
			\begin{pmatrix}0&2I_{n_2}\end{pmatrix}
			 & \sqrt{-1}(Z'+\adj{W})Y_2^{-1}(Z'+W)
		\end{pmatrix} \\
		=    &
		\begin{pmatrix}
			\begin{pmatrix}0&0\\0&\sqrt{-1}\adj{F_3}SF_3\end{pmatrix}
			+\sqrt{-1}\begin{pmatrix}\adj{F_1}F_1&\adj{F_1}F_2\\\adj{F_2}F_1&\adj{F_2}F_2\end{pmatrix}
			 & \begin{pmatrix}0\\2I_{n_2}\end{pmatrix}                            \\
			\begin{pmatrix}0&2I_{n_2}\end{pmatrix}
			 & 4\sqrt{-1} F_3^{-1} \adj{S} (I_{n_2}-S\adj{S})^{-1} \adj{F_3}^{-1}
		\end{pmatrix}.
	\end{align*}
	\normalsize

	For a complex variable \(t\) with \(\abs{t} \leq 1\), we have the Taylor expansion
	\[\hat{F}(tS)=\sum_{\nu=0}^\infty \hat{F}_\nu(S)t^\nu.\]
	Then we have
	\[\hat{F}_\nu(tS)=\hat{F}_\nu(S)t^\nu\]
	and
	\[\quad \hat{F}(S)=\sum_{\nu=0}^\infty \hat{F}_\nu(S).\]
	By substituting \(Se^{\sqrt{-1}\psi}\) with some real number \(\psi\) for \(S\),
	we know that the integral
	\begin{align*}
		\int_{\mathfrak{S}_{n_2}} & \left<\hat{F}_\nu(S), \rho_{n_2}(\,\adj{F_3}(I_{n_2}-\adj{S}S)F_3, \,^t\!F_3(I_{n_2}-^{t}\!{S}\overline{S})\overline{F_3})Q(R_3,\overline{s})\right> \\
		                          & \qquad\cdot\det(I_{n_2}-\adj{S}S)^{\kappa/2-s-2n_2}dS
	\end{align*}
	vanishes unless \(\nu=0\) (See, for example, \cite{Kozima2021Pullback}) and the integral \eqref{eq:int2} is equal to
	\begin{align*}
		  & 2^{|\kappa|-mn_2(2s+2n_2)} \abs{\delta(g_1)\det({Y'_1})^{1/2}}^{\kappa-2s}\overline{\rho'_{n_1}}(M(g_1))^{-1}
		\\
		  & \quad\cdot\int_{\mathfrak{S}_{n_2}} \left<\hat{F}_0(S), \rho_{n_2}(\,\adj{F_3}(I_{n_2}-\adj{S}S)F_3,
		\,^t\!F_3(I_{n_2}-{}^{t}\!{S}\overline{S})\overline{F_3})Q(R_4,\overline{s})\right>                                                                                                                                \\
		  & \qquad \qquad\cdot\det(I_{n_2}-\adj{S}S)^{\kappa/2-s-2n_2}dS                                                                                                                                                   \\
		= & 2^{|\kappa|-mn_2(2s+2n_2)}\abs{\delta(g_1)\det({Y'_1})^{1/2}}^{\kappa-2s}\overline{\rho'_{n_1}}(M(g_1))^{-1}
		\\
		  & \quad\cdot\int_{\mathfrak{S}_{n_2}}
		\left<\rho_{n_2}(\adj{Z'}-Z',\overline{Z'}-^t\!Z')\rho_{n_2}(\adj{Z'}^{-1},\overline{Z'}^{-1})F(\adj{Z'}^{-1})\right.,                                                                                             \\
		  & \qquad \qquad\left.
		\rho_{n_2}(\,\adj{F_3}(I_{n_2}-\adj{S}S)F_3,
		\,^t\!F_3(I_{n_2}-{}^{t}\!{S}\overline{S})\overline{F_3})Q(R_4,\overline{s})\right>                                              \cdot\det(I_{n_2}-\adj{S}S)^{\kappa/2-s-2n_2}dS                                   \\
		= & 2^{|\kappa|-mn_2(2s+2n_2)}(-2\sqrt{-1})^{\abs{\rho_{n_2}}}\abs{\delta(g_1)\det({Y'_1})^{1/2}}^{\kappa-2s}\overline{\rho'_{n_1}}(M(g_1))^{-1}
		\\
		  & \quad\cdot\int_{\mathfrak{S}_{n_2}}\left<\rho_{n_2}((I_{n_2}-\adj{S}S)\adj{F_3}^{-1},(I_{n_2}-{}^{t}\!{S}\overline{S}) {}^{t}\!{F_3}^{-1})\rho_{n_2}(\adj{Z'}^{-1},\overline{Z'}^{-1})F(\adj{Z'}^{-1}),\right. \\
		  & \qquad \qquad \left.\rho_{n_2}(F_3,\overline{F_3})Q(R_4,\overline{s})\right>\cdot\det(I_{n_2}-\adj{S}S)^{\kappa/2-s-2n_2}dS                                                                                    \\
		= & (-\sqrt{-1})^{\abs{\rho_{n_2}}} c(s,\rho_{n_2})\abs{\delta(g_1)\det({Y'_1})^{1/2}}^{\kappa-2s}\overline{\rho'_{n_1}}(M(g_1))^{-1}
		\rho_{n_2}(\adj{Z'}^{-1},\overline{Z'}^{-1})F(\adj{Z'}^{-1}), \tag{5.3}\label{eq:int3}
	\end{align*}
	where
	\begin{align*}
		R_4= &
		\begin{pmatrix}
			\begin{pmatrix}0&0\\0&\sqrt{-1}\adj{F_3}SF_3\end{pmatrix}
			 & \begin{pmatrix}0\\I_{n_2}\end{pmatrix}                            \\
			\begin{pmatrix}0&I_{n_2}\end{pmatrix}
			 & \sqrt{-1} F_3^{-1} \adj{S} (I_{n_2}-S\adj{S})^{-1} \adj{F_3}^{-1}
		\end{pmatrix}.
	\end{align*}

	If we write \(g_1=t_{n_1,n_2}(A_{n_2})\, \mu\, s_{n_1,n_2}(h)\, k\)
	with \(A_{n_2}\in \prod_{\va}\mathrm{GL}_{n_1-n_2}(\bbc)\),
	\(\mu\in U_{n_1,n_2,\infty}\),
	\(h\in G_{n_1,n_2,\infty}\) and \(k\in \calk_{n_1,n_2,\infty}\) by the Iwasawa decomposition,
	we have \(h\left<\bfi_{n_2}\right>=Z'\) and \(h^\natural\left<\bfi_{n_2}\right>=\adj{Z'}^{-1}\).
	Therefore, \eqref{eq:int3} is equal to
	\begin{align*}
		  & (-\sqrt{-1})^{\abs{\rho_{n_2}}}
		c(s,\rho_{n_2})\abs{\delta(g_1)\delta(h)^{-1}}^{\kappa-2s}\overline{\rho'_{n_1}}(M(g_1))^{-1}
		\rho_{n_2}(M(w_0h^\natural))f(h^\natural)                                           \\
		= &
		c(s,\rho_{n_2})\abs{\delta(g_1)\delta(h)^{-1}}^{\kappa-2s}\overline{\rho'_{n_1}}(M(g_1))^{-1}
		\overline{\rho'_{n_2}}(M(h))\det(h)^{\kappa}
		f^\natural(h)                                                                       \\
		= & c(s,\rho_{n_2})\cdot\overline{\epsilon(f^\dagger_\chi)_{n_2}^{n_1}(g_1,s)}.
	\end{align*}
	This concludes the proof.
\end{proof}

\vskip.5\baselineskip

Combining the above local calculations of Proposition~\ref{prop:good}, Proposition~\ref{prop:bad} and Proposition~\ref{prop:arch}
and noting that
\[\left(f,(\dkl E_{n,\kappa})(\iota(g_1,*),\overline{s};\frakn, \chi)\right)
	=\sum_{\gamma_1\in P_{n_1,n_2}^\infty\backslash G_{n_1}^\infty}\left(f,W_{n_2}(\gamma_1g_1,*, s; \chi)\right),
\]
we obtain the main theorem.
\begin{thm}\label{thm:pull}
	Let \(S\) be the set of finite places \(v\) dividing \(\frakn\),
	and take \(s \in \bbc\) such that \(\Re(s)>n\).
	\begin{enumerate}
		\item If \(n_1=n_2\), for any Hecke eigenform \(f\in \mathcal{A}_{0,n_2}(\rho_{n_2},\frakn)\), we have
		      \begin{align*}
			      \left(f,(\dkl E^\theta_{n,\kappa})(\iota(g_1,*),\overline{s}; \chi)\right)
			      =  c(s,\rho_{n_2}) \cdot \prod_{v\mid\frakn}[\calk_{n,v}:\calk_{n,v}(\frakn)]
			      \cdot D_S(s,f;\overline{\chi}) \cdot \overline{f^\dagger_\chi(g_1)}.
		      \end{align*}
		\item If \(\frakn=\inte{\kp}\), for any Hecke eigenform \(f\in \mathcal{A}_{0,n_2}(\rho_{n_2})\), we have
		      \begin{align*}
			      \left(f,(\dkl E_{n,\kappa})(\iota(g_1,*),\overline{s};\frakn, \chi)\right)
			      =  c(s,\rho_{n_2}) \cdot D(s,f;\overline{\chi}) \cdot \overline{[f^\dagger_\chi]_{n_2}^{n_1}(g_1,s; \chi)}.
		      \end{align*}
	\end{enumerate}
	Here a \(\bbc\)-valued function \(c(s,\rho_{n_2})\) is defined in Proposition~\ref{prop:arch}, which does not depend on \(n_1\).
\end{thm}
\bibliography{Pullback}
\bibliographystyle{abbrv}

\end{document}